\documentclass[11pt]{article}

\usepackage{epsfig}
\usepackage{subfigure}
\usepackage{graphicx}
\usepackage{latexsym}
\usepackage{graphicx}
\usepackage{latexsym, bm, amsthm, amsmath}
\usepackage{algorithm}
\usepackage{algorithmic}
\usepackage{xcolor}
\usepackage{multirow,booktabs}
\usepackage{contour,soul,color}
\usepackage[colorlinks=true]{hyperref}
\usepackage{booktabs,cases}
\usepackage{tabularx,multirow} 
\usepackage{cite}
\usepackage{amsfonts}
\usepackage{epstopdf}

\hypersetup{urlcolor=blue,citecolor=blue,linkcolor=blue}

\usepackage{amsmath}
\numberwithin{equation}{section}

\newtheorem{lemma}{Lemma}[section]

\newtheorem{theorem}{Theorem}[section]

\newtheorem{remark}{Remark}[]  

\newsavebox{\tablebox}

\def \[{\begin{equation}}
\def \]{\end{equation}}

\def \[{\begin{equation}}
\def \]{\end{equation}}
\def \<{\begin{eqnarray*}}
\def \>{\end{eqnarray*}}

\begin{document}
\title{{Optimal error estimate of two  linear and momentum-preserving Fourier pseudo-spectral schemes for the RLW equation}}
\author{Qi Hong$^{a}$, ~Yushun Wang$^{b}$, ~Yuezheng Gong$^{c,*}$\\ \\
$^{a}$ Graduate School of China Academy of Engineering Physics, Beijing 100088, China\\
$^{b}$ Jiangsu Key Laboratory for NSLSCS, School of Mathematical Sciences,\\
Nanjing Normal University, Jiangsu 210023, China\\
$^{c}$College of Science, Nanjing University of Aeronautics and Astronautics,
Nanjing, 210016, China
}
\date{}
\maketitle

\begin{abstract}
In this paper, two novel linear-implicit and momentum-preserving Fourier pseudo-spectral schemes are proposed and analyzed for the regularized long-wave equation. The numerical methods are based on the blend of the Fourier pseudo-spectral method in space and the linear-implicit Crank-Nicolson method or the leap-frog scheme in time. The two fully discrete linear schemes are shown to possess the discrete momentum conservation law, and the linear systems resulting from the schemes are proved uniquely solvable. Due to the momentum conservative property of the proposed schemes, the Fourier pseudo-spectral solution is proved to be bounded in the discrete $L^{\infty}$ norm. Then by using the standard energy method, both the linear-implicit Crank-Nicolson momentum-preserving scheme and the linear-implicit leap-frog momentum-preserving scheme are shown to have the accuracy of $\mathcal{O}(\tau^2+N^{-r})$  in the discrete $L^{\infty}$ norm without any restrictions on the grid ratio, where $N$ is the number of nodes and $\tau$ is the time step size. Numerical examples are carried out to verify the correction of the theory analysis and the efficiency of the proposed schemes.\\
\end{abstract}

{\bf Keywords:} regularized long-wave equation, momentum-preserving, linear conservative scheme, Fourier pseudo-spectral method, error estimate.
\begin{figure}[b]
\small \baselineskip=10pt
\rule[2mm]{1.8cm}{0.2mm} \par
$^{*}$Corresponding author.\\
E-mail address: gongyuezheng@nuaa.edu.cn~(Yuezheng Gong).
\end{figure}

\pagestyle{myheadings}
\markboth{\hfil 
   \hfil \hbox{}}
{\hbox{} \hfil 
Qi Hong, Qikui Du and Yushun Wang  \hfil}

\section{Introduction}\label{sec;introduction}
In this paper, we consider the following regularized long-wave (RLW) type equation
\begin{equation}\label{model-eq}
\begin{array}{lll}
&u_t+au_x-\sigma u_{xxt}+\big(F'(u)\big)_x=0,\quad &x\in(x_L,x_R),\; t\in[0,T],\\[0.3cm]
&u(x,0)=u_0(x),\quad &x\in[x_L,x_R],\\[0.3cm]
&u(x,t)=u(x+L,t),\quad &t\in[0,T],
\end{array}
\end{equation}
where $F(u)=\gamma u^3/6$, $L=x_R-x_L$ and $u_0(x)$ is a given function, $a$, $\sigma$ and $\gamma$ are positive constants. The RLW equation was proposed first by Peregrine
\cite{Peregrine1966} and later by Benjamin et al. \cite{Benjamin1972} as a model for small amplitude long waves on the surface of water in a channel. Generalizations such as the generalized RLW equation
or the modified RLW equation \cite{Avrin1985} and generalized Rosenau-Kawhara-RLW equation \cite{Rosenau1985} also arise from various applications. The RLW is very important in physics media since it describes phenomena with weak nonlinearity and dispersion waves, including nonlinear transverse waves in shallow water, ion-acoustic and magneto hydrodynamic waves in plasma and phonon packets in nonlinear crystals.
It admits three conservation laws \cite{Olver1979} given by
\begin{align}\label{conti-invar-eq}
\mathcal{I}_1=\int_{x_L}^{x_R}udx,\quad
\mathcal{I}_2=\int_{x_L}^{x_R}\left(u^2+\sigma u_x^2\right)dx,\quad
\mathcal{I}_3=\int_{x_L}^{x_R}\left(\dfrac{\gamma}{6}u^3+\dfrac{a}{2}u^2\right)dx,
\end{align}
which correspond to mass, momentum and energy of the system, respectively.
Various numerical techniques are applied for the RLW equation, particularly including finite difference scheme \cite{Zhang2005}, the various forms of finite element methods \cite{Danaf2005,Dogan2002,Esen2006,Saka2008}, pseudo-spectral method \cite{Djidjeli2003,Guo1988}, meshless collocation method using radial basis function \cite{Shokri2010}, least square method \cite{Gardner1996,Dag2000,Gu2008} and collocation methods with quadratic B-splines and septic splines \cite{Soliman2001,Dag2004,Soliman2005}, and so on.

In Ref. \cite{Fei1995}, the authors pointed out that the non-conservative schemes may easily induce nonlinear blow-up. Li and Vu-Quoc also said: ``in some areas, the ability to preserve some invariant properties of the original differential equation is a criterion to judge the success of a numerical simulation" \cite{LiVuQuoc1995}. Therefore, for studying long time dynamics of a dynamical system, there has been a surge on constructing numerical methods for dynamical systems governed by differential equations to preserve as many properties of the continuous system as possible. Numerical methods that preserve at least some of the structural properties of the continuous dynamical system are called geometric integrators or structure-preserving algorithms \cite{Feng2003book,Hairer2006book,Bubb2003Book}. Nowadays, a large number of structure-preserving algorithms have been developed for the RLW equation. Sun and Qin \cite{Sun2004} constructed a multi-symplectic Preissman scheme by using the implicit midpoint rule both in space and time. Cai \cite{Cai2009} developed a 6-point multi-symplectic Preissman scheme. An explicit 10-point multi-symplectic Euler-box scheme for the RLW equation was proposed in \cite{02Cai2009}. In \cite{Hong2018,Hong2017,Kong2015,Sun2010,WangTC2013}, some methods that conserve energy conservation laws were developed. Cai and Hong \cite{CaiHong2017} proposed three local energy-preserving algorithms for the RLW-type equation.

Compared with the numerical application of the RLW equation, there exists few literatures about the convergence analysis. Solan \cite{Sloan1991} investigated the RLW equation by a three-level explicit Fourier pseudo-spectral scheme. But the stability and error estimate were not pursued. Coupled with the Richardson extrapolation, Zheng et al. \cite{Zheng2013} proposed and analyzed a two-level nonlinear Crank-Nicolson finite difference scheme for the RLW equation, where the accuracy of $\mathcal{O}(\tau^2+h^4)$ of their method was obtained. In \cite{Kang2015}, Kang et al. presented a second-order in time linearized semi-implicit Fourier pseudo-spectral scheme for the generalized RLW equation. They showed that such an approximate solution satisfies $\mathcal{O}(\tau^2)$ in time and a spectral accuracy in space by assuming the numerical solution bounded in $L^{\infty}$ norm. In \cite{Caij2017}, Cai et al. proposed two explicit local momentum-preserving schemes and two fully implicit local momentum-preserving schemes and gave the error estimates in $L^{\infty}$ norm for their proposed implicit schemes. There is no doubt that a scheme with adequate theoretical foundations is more competitive and reliable in practical applications.

In this paper, we aim to develop linear structure-preserving algorithms for the RLW equation. We first start from an equivalent from of the RLW equation and  discretize it by the Fourier pseudo-spectral method in space to arrive at a semi-discrete ordinary differential equation (ODE) system, where the momentum is conserved in the spatial semi-discrete level. Then we respectively apply the linear-implicit Crank-Nicolson scheme and the leap-frog scheme in time for the ODE system to obtain two fully discrete linear schemes. The two proposed schemes are then shown to satisfy a fully discretized momentum conservation law and be uniquely solvable. According to the equivalence between the semi-norms induced by the Fourier pseudo-spectral method and the finite difference method \cite{GongJCP2017} and the discrete momentum conservation law, the numerical solution is proved to be bounded in the discrete $L^{\infty}$ norm. Then by the standard energy method, the linear-implicit Crank-Nicolson momentum-preserving scheme is proved to has the accuracy of $\mathcal{O}(\tau^2+N^{-r})$ in the discrete $L^{\infty}$ norm without imposing any constraints on the grid ratio. And the linear-implicit leap-frog momentum-preserving scheme can be similarly discussed. Finally, some numerical examples are presented to demonstrate the correction of the theory analysis and the efficiency of the
proposed schemes.

In summary, the proposed methods have the following advantages:
\begin{itemize}
	\item
	The schemes preserve the discrete momentum conservation law, which implies that they possess excellent stability.
	\item
	One only needs to solve a linear equation system at each time step, which reduces the computational cost.
	\item
	High order, i.e. they are second order in time and spectral accuracy in space.
	\item
	The convergence results of the two schemes are rigorously analyzed without any constraints on the grid ratio.
\end{itemize}

The remainder of the paper is organized as follows. In section \ref{semi-discrete-system}, we apply the Fourier pseudo-spectral method in space for the RLW equation, which satisfies the semi-discrete momentum conservation law. In section \ref{full-discrete-system}, we respectively employ the linear-implicit Crank-Nicolson method and the leap-frog method in time to obtain two fully discrete linear conservative schemes, where their momentum conservative property and unique solvability are proved rigorously. The convergence results are obtained in section \ref{conver-results}. In section \ref{Numer-exper}, numerical experiments are presented to illustrate the efficiency and accuracy of the proposed
methods. Finally, we give conclusions and further comments.

\section{Structure preserving spatial discretization}\label{semi-discrete-system}
In this section, we devise a Fourier pseudo-spectral spatial discretization for the RLW equation with periodic boundary condition. The semi-discrete scheme is shown to preserve the corresponding momentum conservation law.

First, we introduce some notations and useful lemmas. Let $N$ be a positive even integer. The domain $\Omega=[x_L,x_R]$ is uniformly partitioned with mesh size $h=(x_R-x_L)/N$ and
$\Omega_h=\{x_j|x_j=x_L+jh,\ 0\leq j\leq N-1\}$.
Let $V_{h} = \big\{u|u=\{u_{j}|x_{j}\in
\Omega_{h}\} \big\}$ be the space of grid functions on $\Omega_h$. Throughout this paper, the hollow letters ${\mathbb A}, {\mathbb B}, {\mathbb D}, \cdots$  will be used to denote rectangular matrices with a number of columns greater than one, while the bold ones ${\mathbf U}, {\mathbf V}, {\mathbf W}, \cdots$ will represent vectors.
For any two grid functions $\mathbf{U},\ \mathbf{V}\in V_h$, we define the discrete inner product
\begin{equation*}
(\mathbf{U},\mathbf{V})_h=h\sum_{j=0}^{N-1}U_j\overline{V}_j,
\end{equation*}
where $\overline{V}_j$ denotes the conjugate of $V_j$.  The discrete norms of $\mathbf{U}$ and its difference quotient are defined, respectively, as
\begin{equation*}
\|\mathbf{U}\|_{h}=\sqrt{(\mathbf{U},\mathbf{U})_h},\quad \|\delta_x^+\mathbf{U}\|_{h}=\sqrt{(\delta_x^+\mathbf{U},\delta_x^+\mathbf{U})_h},\quad
\|\mathbf{U}\|_{\infty,h}=\max_{0\leq j\leq N-1}|U_j|,
\end{equation*}
where $\delta_x^+ U_j = (U_{j+1}-U_j)/h.$ It is easy to prove that
\begin{align*}
\|\delta_x^+\mathbf{U}\|_{h}=\sqrt{(-\mathbb{A}_2\mathbf{U},\mathbf{U})_h},
\end{align*}
where
\begin{align*}
\mathbb{A}_2=\dfrac{1}{h^2}
\left[
\begin{array}{rrrrrr}
-2&\ 1&\ 0&\ 0&\ \cdots&\ 1\\
1&\ -2&\ 1&\ 0&\ \cdots&\ 0\\
0&\ 1&\ -2&\ 1&\ \cdots& 0\\
\quad& \quad&\ \ddots&\ \ddots&\ \ddots  \\
0&\ \cdots&\  0&\  1&\  -2&\  1\\
1&\ \cdots&\  0&\  0&\  1&\  -2\\
\end{array}
\right].
\end{align*}

We define \cite{ChenQin2001,ShenTang2006}
\begin{align*}
S^{'}_N=\mathrm{span}\{g_j(x),~j=0,1,\ldots,N-1 \}
\end{align*}
as the interpolation space, where $g_j(x)$ is trigonometric polynomial of degree $N/2$ given by
\begin{align*}
g_j(x)=\dfrac{1}{N}\sum_{k=-N/2}^{N/2}\dfrac{1}{c_k}e^{ik\mu(x-x_j)},
\end{align*}
where  $c_l=1 (|l|\neq N/2)$, $c_{-N/2}=c_{N/2}=2$ and $\mu=2\pi/(x_R-x_L)$. We define the interpolation operator $I_N:C(\Omega)\rightarrow S^{'}_N$
\begin{align}\label{interpol-eq}
I_Nu(x)=\sum_{j=0}^{N-1}u_jg_j(x),
\end{align}
where $u_j=u(x_j,t)$. To obtain derivative $\partial_{x}^{k}I_{N}u(x)$ at collocation points, we differentiate \eqref{interpol-eq} and evaluate the resulting expressions at point
$x_{j}$:
\begin{align*}
\dfrac{\partial^kI_Nu(x_j)}{\partial x^k}=\sum_{l=0}^{N-1}u_l\dfrac{d^kg_l(x_j)}{dx^k}=\sum_{l=0}^{N-1}(\mathbb{D}_{k})_{jl}u_l,
\end{align*}
where $\mathbb{D}_{k}$ is a so-called $k$-order differential matrix \cite{ShenTang2006}.

\begin{lemma}[\cite{Gong2014CICP}] \label{lem-FFT} Let
	\begin{align*}
	\Lambda_k=
	\begin{cases}
	\left[i\mu\mathrm{diag}(0,1,\cdots,\dfrac{N}{2}-1,0,-\dfrac{N}{2}+1,\cdots,-1)\right]^k,&\quad\mathrm{k\ odd},\\[0.3cm]
	\left[i\mu\mathrm{diag}(0,1,\cdots,\dfrac{N}{2}-1,\dfrac{N}{2},-\dfrac{N}{2}+1,\cdots,-1)\right]^k,&\quad\mathrm{k\ even},
	\end{cases}
	\end{align*}
	we have
	\begin{align*}
	\mathbb{D}_{k}=F_N^{-1}\Lambda_k F_N,
	\end{align*}
	where $F_N$ is the discrete Fourier transform, and $F_N^{-1}$ is the discrete inverse Fourier transform.
\end{lemma}

\begin{remark}
	With the help of Lemma \ref{lem-FFT}, we can evaluate the derivatives by using the FFT algorithm instead of the spectral differentiation matrix.
\end{remark}

Here, we define a new semi-norm as follows:
\begin{align}\label{new-semi-norm-def}
|\mathbf{U}|_{h}=\sqrt{(-\mathbb{D}_2\mathbf{U},\mathbf{U})_h}, ~ \mathbf{U}\in V_h.
\end{align}
Note that $\mathbb{D}_2$ is real symmetric and negative semi-definite, so the definition \eqref{new-semi-norm-def} is meaningful. Next, we have the following lemma.

\begin{lemma}[\cite{GongJCP2017}]\label{semi-norm-equiv}
	For any grid function $\mathbf{U}\in V_h$, we have
	\begin{align}
	&\|\mathbb{D}_1\mathbf{U}\|_{h}\leq |\mathbf{U}|_h,\label{semi-norm-equiv-ineq2}\\[0.3cm]
	&\|\delta^+_x\mathbf{U}\|_{h}\leq |\mathbf{U}|_{h} \leq \dfrac{\pi}{2}\|\delta^+_x\mathbf{U}\|_{h}.\label{semi-norm-equiv-ineq1}
	\end{align}
\end{lemma}

\begin{remark}
	Lemma \ref{semi-norm-equiv}  indicates that the semi-norm induced by the Fourier pseudo-spectral method is equivalent to that of the finite difference method, which will play an important role in the proof of boundedness of the numerical solution.
\end{remark}

We next discuss how to design momentum-preserving spatial discretization for the RLW equation. To this end, we rewrite the RLW equation into the following equivalent form
\begin{align}\label{RLW-equivalent-eq}
u_t-\sigma u_{xxt}+au_x+\dfrac{\gamma}{3}(u\partial_x+\partial_xu)u=0.
\end{align}
Applying the Fourier pseudo-spectral method in space for \eqref{RLW-equivalent-eq}, we obtain a semi-discrete system
\begin{align}\label{semi-discr-syst1-equiv}
(\mathbb{I}-\sigma\mathbb{D}_2)\dfrac{d}{dt}\mathbf{U}+\mathbb{D}(\mathbf{U})\mathbf{U}=0,
\end{align}
where $\mathbb{D}(\mathbf{U})$ is defined as
\begin{align*}
\mathbb{D}(\mathbf{U}) = a\mathbb{D}_1 + \dfrac{\gamma}{3}\big(\mathrm{diag}(\mathbf{U})\mathbb{D}_1 + \mathbb{D}_1\mathrm{diag}(\mathbf{U})\big).
\end{align*}
Note that $\mathbb{D}(\mathbf{U})$ is anti-symmetric for any $\mathbf{U}$ because of the anti-symmetry of $\mathbb{D}_1$. Next we will present that the semi-discrete system \eqref{semi-discr-syst1-equiv} possesses the discrete momentum conservation law.

\begin{theorem}
	The semi-discrete scheme \eqref{semi-discr-syst1-equiv}  preserves the discrete momentum conservation law
	\begin{align*}
	\dfrac{d}{dt}\mathcal{I}_{2_h} = 0,
	\end{align*}
	where $\mathcal{I}_{2_h}=\|\mathbf{U}\|_{h}^2+\sigma|\mathbf{U}|_{h}^2.$
\end{theorem}
\begin{proof}
	Noticing the anti-symmetric property of $\mathbb{D}(\mathbf{U})$, we obtain
	\begin{align*}
	\Big(\mathbb{D}(\mathbf{U})\mathbf{U},\mathbf{U}\Big)_h = 0.
	\end{align*}
	Taking the discrete inner product of \eqref{semi-discr-syst1-equiv} with 2$\mathbf{U}$, we deduce
	\begin{align*}
	\dfrac{d}{dt}(\|\mathbf{U}\|_{h}^2+\sigma|\mathbf{U}|_{h}^2) = 0.
	\end{align*}
	This completes the proof.
\end{proof}

\section{Fully discrete linear-implicit momentum-preserving scheme}\label{full-discrete-system}
In this section, we introduce two temporal methods for the semi-discrete system \eqref{semi-discr-syst1-equiv} to arrive at fully discretized schemes. One is the linear-implicit Crank-Nicolson method and the other is the leap-frog method, which both preserve the fully discrete momentum conservative law. For ease of reading, we call them LCN-MP and LLF-MP, respectively.

\subsection{Linear-implicit Crank-Nicolson scheme}
For a positive integer $N_t$, we denote time-step $\tau=T/N_t$, $t_n=n\tau,\ 0\leq n\leq N_t$. We define
\begin{align*}
&\delta^+_t{\mathbf U}^n=\dfrac{\mathbf U^{n+1}-\mathbf U^{n}}{\tau},\quad \widehat{\mathbf{U}}^{n+\frac{1}{2}}=\dfrac{3\mathbf U^{n}-\mathbf{U}^{n-1}}{2},\quad
\mathbf{U}^{n+\frac{1}{2}}=\dfrac{\mathbf{U}^{n+1}+\mathbf{U}^n}{2}.
\end{align*}
In this paper, we denote the numerical solution $U_j^n\approx u(x_j,t_n)$ and $C$ denotes a positive constant which is independing of mesh grid and may be different in different cases.

Applying the linear-implicit Crank-Nicolson scheme in time for the semi-discrete system \eqref{semi-discr-syst1-equiv}, we obtain LCN-MP as follows
\begin{align}\label{full-discr-LICN}
(\mathbb{I}-\sigma \mathbb{D}_2) \delta^+_t\mathbf{U}^n + \mathbb{D}(\widehat{\mathbf{U}}^{n+\frac{1}{2}}) \mathbf{U}^{n+\frac{1}{2}} = 0,
\end{align}
where $n\geq 1$ and $\mathbf{U}^1$ is the solution of the following equation
\begin{align}\label{order-2-scheme}
(\mathbb{I}-\sigma\mathbb{D}_2) \delta_t^+\mathbf{U}^0 + \mathbb{D}(\mathbf{U}^{\frac{1}{2}}) \mathbf{U}^{\frac{1}{2}} = 0.
\end{align}
Next, we prove that LCN-MP conserves the discrete momentum conservation law and is uniquely solvable.
\begin{theorem}\label{t-S1-G2-LICN}
	LCN-MP \eqref{full-discr-LICN} with \eqref{order-2-scheme} satisfies the following discrete momentum conservation law
	\begin{align}\label{full-G1-result2}
	{\mathcal{I}_{2_h}}^{n} \equiv {\mathcal{I}_{2_h}}^0,\quad \forall\ n\geq 0,
	\end{align}
	where ${\mathcal{I}_{2_h}}^n=\|\mathbf{U}^n\|_{h}^2+\sigma|\mathbf{U}^n|_{h}^2.$ 	
\end{theorem}
\begin{proof}
	Noticing that $\mathbb{D}(\mathbf{U})$ is anti-symmetric for any $\mathbf{U},$ we have
	\begin{align*}
	\Big(\mathbb{D}(\widehat{\mathbf{U}}^{n+\frac{1}{2}})\mathbf{U}^{n+\frac{1}{2}},\mathbf{U}^{n+\frac{1}{2}}\Big)_h=0.
	\end{align*}
	Therefore, taking the discrete inner product of \eqref{full-discr-LICN} with $2 \mathbf{U}^{n+\frac{1}{2}}$,  we have
	\begin{align*}
	0 = \Big((\mathbb{I}-\sigma \mathbb{D}_2)\delta^+_t\mathbf U^{n},2\mathbf{U}^{n+\frac{1}{2}}\Big)_h = \dfrac{1}{\tau}\Big(\|\mathbf{U}^{n+1}\|^2_{h}+\sigma|\mathbf{U}^{n+1}|^2_{h} -\|\mathbf{U}^{n}\|^2_{h}-\sigma|\mathbf{U}^{n}|^2_{h}\Big),
	\end{align*}
	which implies  that
	\begin{align}\label{theo1-eq1}
	{\mathcal{I}_{2_h}}^{n+1} = {\mathcal{I}_{2_h}}^{n},\quad\forall\;n\geq 1.
	\end{align}
	Similarly, it follows from \eqref{order-2-scheme} that
	\begin{align}\label{theo1-eq2}
	{\mathcal{I}_{2_h}}^{1} = {\mathcal{I}_{2_h}}^{0}.	
	\end{align}
	Combining \eqref{theo1-eq1} and \eqref{theo1-eq2} leads to \eqref{full-G1-result2}. This completes the proof.
\end{proof}

\begin{theorem}\label{t-S1-unique}
	For any $\sigma>0,$ LCN-MP \eqref{full-discr-LICN} is uniquely solvable.
\end{theorem}
\begin{proof}
	The scheme \eqref{full-discr-LICN} can be written as the following linear equation system
	\begin{align*}
	\mathbb{B} \mathbf{U}^{n+1} = \mathbf{b},
	\end{align*}
	where $\mathbb{B} = \mathbb{I}-\sigma\mathbb{D}_2 + \frac{\tau}{2} \mathbb{D}(\widehat{\mathbf{U}}^{n+\frac{1}{2}})$ and $\mathbf{b} = \left(\mathbb{I}-\sigma\mathbb{D}_2 - \frac{\tau}{2} \mathbb{D}(\widehat{\mathbf{U}}^{n+\frac{1}{2}})\right) \mathbf{U}^{n}.$ In order to obtain the unique solvability of the scheme, we need to prove that the matrix $\mathbb{B}$ is invertible.
	
	If $\mathbb{B} \mathbf{x} = \mathbf{0},$ then we have
	\begin{align*}
	0 = \mathbf{x}^T \mathbb{B} \mathbf{x} = \mathbf{x}^T (\mathbb{I}-\sigma\mathbb{D}_2)\mathbf{x},
	\end{align*}
	where the anti-symmetry of $\mathbb{D}(\mathbf{U})$ was used. Note that $\mathbb{I}-\sigma\mathbb{D}_2$ is symmetric positive definite for $\sigma>0$, thus $\mathbf{x} = \mathbf{0},$ i.e. $\mathbb{B} \mathbf{x} = \mathbf{0}$ has only zero solution. Therefore, $\mathbb{B}$ is invertible. This completes the proof.
\end{proof}

\subsection{Leap-frog scheme}
Denote
\begin{align*}
&\delta_t{\mathbf U}^n=\dfrac{\mathbf U^{n+1}-\mathbf U^{n-1}}{2\tau},\quad \widehat{\mathbf{U}}^n=\dfrac{\mathbf U^{n+1}+\mathbf{U}^{n-1}}{2}.
\end{align*}
Applying the leap-frog scheme in time for the semi-discrete system \eqref{semi-discr-syst1-equiv}, we obtain  LLF-MP
\begin{align}\label{full-discr-LILF}
(\mathbb{I}-\sigma\mathbb{D}_2) \delta_t\mathbf{U}^n + \mathbb{D}\left(\mathbf{U}^{n}\right) \hat{\mathbf{U}}^{n} = 0, \quad n\geq 1.
\end{align}
Here we still choose \eqref{order-2-scheme} to compute the initial datum for the second level values of the three time levels scheme \eqref{full-discr-LILF}.

\begin{theorem}\label{t-S1-G2-LILF}
	LLF-MP \eqref{full-discr-LILF} with \eqref{order-2-scheme} satisfies the following discrete momentum conservation law
	\begin{align*}
	{\mathcal{I}_{2_h}}^{n} \equiv {\mathcal{I}_{2_h}}^0,\quad \forall\ n\geq 0,
	\end{align*}
	where ${\mathcal{I}_{2_h}}^n=\|\mathbf{U}^n\|_{h}^2+\sigma|\mathbf{U}^n|_{h}^2.$ 	
\end{theorem}
\begin{proof}
	The proof is analogous to that of Theorem \ref{t-S1-G2-LICN} and thus omitted here.
\end{proof}

\begin{theorem}\label{t-S2-unique}
	LLF-MP \eqref{full-discr-LILF} is uniquely solvable.
\end{theorem}
\begin{proof}
	The proof is similar to Theorem \ref{t-S1-unique} and is thus omitted.
\end{proof}

Both the schemes LCN-MP \eqref{full-discr-LICN} and LLF-MP \eqref{full-discr-LILF} are second order in time and high order in space. The two schemes are linear-implicit, which implies they are very cheap in the numerical calculation.  In what follows,   we mainly show the
analysis for LCN-MP by the standard energy method while the error estimate of LLF-MP can be obtained similarly and thus is omitted.

\section{Prior estimate and convergence analysis}\label{conver-results}
In this section, we analyze the error estimate of LCN-MP  in detail, while LLF-MP can be similarly discussed. Similar to finite element analysis, error estimate of pseudo-spectral scheme relies on the interpolation and the projection theory. We first introduce several notations and some basis results.

Let $C_p^{\infty}(\Omega)$ be a set of infinitely differentiable functions with period $L$, defined on $\mathbb{R}$, and $H_p^r(\Omega)$ is the closure of $C_p^{\infty}(\Omega)$ in $H^r(\Omega)$.
Let $\Omega=[a,b]$, $L^2(\Omega)$ with the inner product $(\cdot,\cdot)$ and the term $\|\cdot\|$. For any positive integer $r$, the semi-norm and the norm of $H^r(\Omega)$ are denoted by $|\cdot|_r$ and $\|\cdot\|_r$, respectively. In this section, $\|\cdot\|_0$ is denoted by $\|\cdot\|$ for simplicity.
For even $N$, we defined the projection space $S_N$ and the interpolation space $S_N^{'}$, respectively,
\begin{align*}
S_N=\left\{u:u(x)=\sum_{|k|\leq N/2}\hat{u}_ke^{ik\mu(x-a)}\right\},\quad
S^{'}_N=\left\{u:u(x)={\sum_{|k|\leq N/2}}^{''}\hat{u}_ke^{ik\mu(x-a)},\hat{u}_{-N/2}=\hat{u}_{N/2}\right\},
\end{align*}
where the summation $\sum^{''}$ is defined by
\begin{align*}
{\sum_{|k|\leq N/2}}^{''}\phi_k=\dfrac{1}{2}\phi_{-\frac{N}{2}}+\sum_{|k|< N/2}\phi_k+\dfrac{1}{2}\phi_{\frac{N}{2}}.
\end{align*}
\begin{remark}
	It is shown easily that
	\begin{align*}
	&S_N^{'}\subseteq S_N,\quad S_{N-2}\subseteq S_N^{'},\\[0.3cm]
	&P_Nu=u,\quad \forall\; u\in S_N,\\[0.3cm]
	&I_Nu=u,\quad \forall\; u\in S^{'}_N,\\[0.3cm]
	&P_N\partial_{x}u=\partial_{x}P_Nu,\quad I_N\partial_{x}u\neq\partial_{x}I_Nu,
	\end{align*}
	where $P_N:L^2(\Omega)\rightarrow S_N$ denotes the orthogonal projection operator and
	$I_N:C(\Omega)\rightarrow S^{'}_N$ denotes the interpolation operator.
\end{remark}
Next, we will introduce some useful lemmas, which play an important role in the proof of the convergence.

\begin{lemma}[\cite{GongJCP2017}]\label{norm-equality}
	For any function $u\in S_N^{'}$, we have $\|u\|\leq \|u\|_{h}\leq \sqrt{2}\|u\|$.
\end{lemma}
\begin{lemma}[\cite{Canuto1982}]\label{approximation-lemma}
	If $0\leq l\leq r$ and $u\in H_p^r(\Omega)$, then
	\begin{align}
	&\|P_Nu-u\|_l\leq CN^{l-r}|u|_r,\label{appro-lem-eq1}\\[0.3cm]
	&\|P_Nu\|_l\leq C\|u\|_l,\label{appro-lem-eq2}
	\end{align}
	in addition, if $r>1/2$, we have
	\begin{align}
	&\|I_Nu-u\|_l\leq CN^{l-r}|u|_r,\label{appro-lem-eq3}\\[0.3cm]
	&\|I_Nu\|_l\leq C\|u\|_l.\label{appro-lem-eq4}
	\end{align}
\end{lemma}
\begin{lemma}[\cite{GongJCP2017}]\label{pro-error-estimate}
	For $u\in H^r_p(\Omega),\ r>1$, let $u^{*}=P_{N-2}u$, then $\|u^{*}-u\|_{h}\leq CN^{-r}|u|_r$.
\end{lemma}

\begin{lemma}\label{semi-pro-error-estimate}
	For $u\in H^{r+1}_p(\Omega)$, $r>1/2$, let  $u^{*}=P_{N-2}u$, then
	\begin{align}
	&|u^*-u|_{h}\leq CN^{-r}|u|_{r+1}, \label{semi-norm-projec-eq}\\[0.3cm]
	&\|\partial_x(I_Nu-u)\|_h\leq CN^{-r}|u|_{r+1}.\label{semi-norm-projec-eq2}
	\end{align}
\end{lemma}
\begin{proof}
	Since
	\begin{align}\label{projec-semi-esti-eq1}
	\begin{split}
	|u^*-u|^2_{h}&=(-\mathbb{D}_2(u^*-u),u^*-u)_h\\[0.3cm]
	&\leq\|-\mathbb{D}_2(u^*-u)\|_{h}\|u^*-u\|_{h}\\[0.3cm]
	&=\|\partial_{xx}I_N(u^*-u)\|_{h}\|u^*-u\|_{h}.
	\end{split}
	\end{align}
	We remark that
	\begin{align}\label{projec-semi-esti-eq2}
	\begin{split}
	\|\partial_{xx}(I_N(u^*-u))\|_{h}&=\|I_N(\partial_{xx}(I_N(u^*-u)))\|_{h}\\[0.3cm]
	&\leq\sqrt{2}\|I_N(\partial_{xx}(I_N(u^*-u)))\| \\[0.3cm]
	&\leq C\|\partial_{xx}(I_N(u^*-u))\|\\[0.3cm]
	&\leq C\|I_N(u^*-u)\|_{2}\\[0.3cm]
	&\leq C\|u^*-u\|_2\leq CN^{1-r}|u|_{r+1},
	\end{split}
	\end{align}
	where the first inequality follows from Lemma \ref{norm-equality}, the second  and the fourth inequality follow from  \eqref{appro-lem-eq4} and the last inequality follows from \eqref{appro-lem-eq1}.
	Substituting \eqref{projec-semi-esti-eq2} into \eqref{projec-semi-esti-eq1} and using Lemma \ref{pro-error-estimate} leads to \eqref{semi-norm-projec-eq}.
	
	Using  \eqref{appro-lem-eq1}, \eqref{appro-lem-eq3} and Lemma \ref{norm-equality} once again, we can easily deduce
	\begin{align*}
	\|\partial_x(I_Nu-u)\|_h&=\|I_N[\partial_x(I_Nu-u)]\|_h\leq\sqrt{2}\|I_N[\partial_x(I_Nu-u)]\| \\[0.3cm]
	&\leq C\|\partial_x(I_Nu-u)\|\leq C\|I_Nu-u\|_1\leq CN^{-r}|u|_{r+1}.
	\end{align*}
	The proof is completed.
\end{proof}

\begin{lemma}[Discrete Sobolev inequality \cite{ZYL1990}] \label{Discr-Sobol-ineq}
	For any discrete functions $\mathbf{U}\in V_h$, there exists
	\begin{align*}
	\|\mathbf{U}\|^2_{\infty,h}\leq 2\|\mathbf{U}\|_{h}\cdot \|\delta_x^+\mathbf{U}\|_{h}+\dfrac{\|\mathbf{U}\|_{h}^2}{L}.
	\end{align*}
\end{lemma}

\begin{lemma}[Discrete Gronwall inequality \cite{ZYL1990}]\label{discr-Gron-ineq}
	Suppose that the nonnegative discrete function $\{\omega^n|n=0,1,2,\cdots,N_t;\; N_t\tau=T\}$ satisfies the inequality
	\begin{align*}
	\omega^n\leq A+B\tau \sum_{k=1}^{N_t}\omega^{k},\quad 1 \leq n\leq N_t,
	\end{align*}
	where $A$ and $B$ are nonnegative constants. Then
	\begin{align*}
	\max_{1\leq n\leq N_t}|\omega^n|\leq Ae^{2BT},
	\end{align*}
	where $\tau$ is sufficiently small, such that $B\tau\leq 1/2$.
\end{lemma}

\subsection{Prior estimate}
\begin{theorem}\label{prior-estimate}
	Assume that the initial condition $u_0(x)=u(x,0)\in H_p^1=\{u(x)\in H^1:u(x)=u(x+L)\}$, then we have the following prior estimates
	\begin{align*}
	\|u\|_0\leq C,\quad \|u_x\|_0\leq C,\quad    \|u\|_{\infty}\leq C,
	\end{align*}
	for the exact solution of the RLW equation \eqref{model-eq} and the prior estimates
	\begin{align*}
	\|\mathbf{U}^n\|_{h}\leq C,\quad |\mathbf{U}^n|_{h}\leq C,\quad    \|\mathbf{U}^n\|_{\infty,h}\leq C,
	\end{align*}
	for the numerical solution of the scheme \eqref{semi-discr-syst1-equiv}.
\end{theorem}
\begin{proof}
	By the continuous invariant $\mathcal{I}_2$ in \eqref{conti-invar-eq}, it is easy to prove that
	\begin{align*}
	\|u\|_0\leq C,\quad \|u_x\|_0\leq C.
	\end{align*}
	It follows from the Sobolev inequality that  $\|u\|_{\infty}\leq C$.
	
	Similarly, the discrete momentum conservation law in Theorem \ref{t-S1-G2-LICN} implies
	\begin{align*}
	\|\mathbf{U}^n\|_{h}\leq C,\quad |\mathbf{U}^n|_{h}\leq C.
	\end{align*}
	Then noticing \eqref{semi-norm-equiv-ineq1}, we get
	\begin{align*}
	\|\delta_x^+\mathbf{U}^n\|_{h}\leq C.
	\end{align*}
	Using Lemma \ref{Discr-Sobol-ineq} yields  $\|\mathbf{U}^n\|_{\infty,h}\leq C$. This completes the proof.
\end{proof}

\subsection{Convergence analysis}
For simplicity, we denote $u^n_{j} = u(x_j, t_n)$ and $U^n_j$ as the exact value of $u(x, t)$ and its numerical approximation at $(x_j, t_n)$, respectively, and set $f(u)=\frac{\gamma}{3}u \partial_xu$, $g(u)=\frac{\gamma}{3}\partial_x(u\cdot u)$. Then
the RLW equation \eqref{RLW-equivalent-eq} can be written as
\begin{align}\label{RLW-new-eq}
u_t-\sigma u_{xxt}+au_x+f(u)+g(u)=0.
\end{align}
Denote
\begin{align*}
u^*=P_{N-2}u,\quad f^*(u)=P_{N-2}f(u),\quad  g^*(u)=P_{N-2}g(u).
\end{align*}
In order to prove the error estimate, we define the local truncation $\xi_j^n$ as follows
\begin{align}\label{trunc-eq}
\xi_j^n = \delta^+_t(u^*)_j^{n} - \sigma\delta^+_t(\mathbb{D}_2u^*)_j^{n} + a(\mathbb{D}_1u^*)_j^{n+\frac{1}{2}} + (f^*(u))^{n+\frac{1}{2}}_j + (g^*(u))^{n+\frac{1}{2}}_j.
\end{align}

\begin{lemma}\label{trunc-lem}
	If $u\in C^3(0,T;H_p^{r}(\Omega)),\; r>1/2$, we have
	\begin{align*}
	|\xi_j^n|\leq C\tau^2,\quad n=0,1,2,\cdots,N_t.
	\end{align*}
\end{lemma}
\begin{proof}
	The projection equation of \eqref{RLW-new-eq} is
	\begin{align*}
	u_t^*-\sigma u^*_{xxt}+au_x^*+f^*(u)+g^*(u)=0.
	\end{align*}
	Note that
	\begin{align*}
	u^*\in S_N^{'},\quad \partial_{x}u^*(x_j,t_{n})=(\mathbb{D}_1u^*)_j^{n},\quad \partial_{xx}u^*(x_j,t_{n})=(\mathbb{D}_2u^*)_j^{n}.
	\end{align*}
	Thus we have
	\begin{align}\label{differ-trunc-eq}
	\xi_j^n = \bigg(\delta^+_t(u^*)_j^{n} - \partial_t(u^*)_j^{n+\frac{1}{2}}\bigg) - \sigma\bigg(\delta^+_t(u_{xx}^*)_j^{n} - \partial_t(u^*_{xx})_j^{n+\frac{1}{2}}\bigg).
	\end{align}
	Using the Taylor expansion, we have
	\begin{align*}
	&(u^{*})_j^{n+1}=(u^{*})_j^{n+\frac{1}{2}}+\dfrac{\tau}{2}\partial_t(u^{*})_j^{n+\frac{1}{2}}+\dfrac{\tau^2}{8}\partial_{tt}(u^{*})_j^{n+\frac{1}{2}}+\mathcal{O}(\tau^3),\\[0.3cm]
	&(u^{*})_j^{n}=(u^{*})_j^{n+\frac{1}{2}}-\dfrac{\tau}{2}\partial_t(u^{*})_j^{n+\frac{1}{2}}+\dfrac{\tau^2}{8}\partial_{tt}(u^{*})_j^{n+\frac{1}{2}}+\mathcal{O}(\tau^3),
	\end{align*}
	which implies
	\begin{align*}
	&\delta^+_t(u^*)_j^{n}=\dfrac{1}{\tau}\left((u^*)_j^{n+1}-(u^*)_j^{n}\right)=\partial_t(u^{*})_j^{n+\frac{1}{2}}+\mathcal{O}(\tau^2),\\[0.3cm]
	&\delta^+_t(u_{xx}^*)_j^{n}=\dfrac{1}{\tau}\left((u_{xx}^*)_j^{n+1}-(u_{xx}^*)_j^{n}\right)=\partial_t(u_{xx}^{*})_j^{n+\frac{1}{2}}+\mathcal{O}(\tau^2).
	\end{align*}
	Substituting the above results into \eqref{differ-trunc-eq}, we arrive at
	\begin{align*}
	|\xi_j^n|\leq C\tau^2,\quad  n=0,1,2,\cdots,N_t-1.
	\end{align*}
\end{proof}

\begin{lemma}\label{start-step-lem}
	Assume that the exact solution $u(x,t)$ of problem \eqref{model-eq} satisfies
	\begin{align*}
	u(x,t)\in C^3(0,T; H^{r+1}_p(\Omega)),\;  r>\dfrac{1}{2},
	\end{align*}
	and  $\mathbf{U}^1$ is the numerical solution of \eqref{order-2-scheme}. Then we have
	\begin{align*}
	\|(u^*)^1-\mathbf{U}^1\|_{h}+\sigma|(u^*)^1-\mathbf{U}^1|_{h}\leq C(\tau^2+N^{-r}).
	\end{align*}
\end{lemma}
\begin{proof}
	Denote $e_j^0=(u^*)_j^0-U_j^0$ and $e_j^1=(u^*)_j^1-U_j^1$. Subtracting \eqref{order-2-scheme} from \eqref{trunc-eq} at $n=0$ leads to
	\begin{align}\label{order2-scheme-error-eq}
	\xi^0=(\mathbb{I}-\sigma\mathbb{D}_2)\delta_t^+e^0+a\mathbb{D}_1e^{\frac{1}{2}}+(F_{\delta})^{\frac{1}{2}}+(G_{\delta})^{\frac{1}{2}},
	\end{align}
	where
	\begin{align*}
	&(F_{\delta})_j^{\frac{1}{2}}=\big(f^*(u)\big)^{\frac{1}{2}}_j-F(U_j^{0},U_j^1),\quad F(U_j^{0},U_j^1)=\dfrac{\gamma}{3}{U}_j^{\frac{1}{2}}(\mathbb{D}_1\mathbf{U}^{\frac{1}{2}})_j,\\[0.3cm]
	&(G_{\delta})_j^{\frac{1}{2}}=\big(g^*(u)\big)^{\frac{1}{2}}_j-G(U_j^{0},U_j^1),\quad G(U_j^{0},U_j^1)=\dfrac{\gamma}{3}\big(\mathbb{D}_1({\mathbf{U}}^{\frac{1}{2}}\odot{\mathbf{U}}^{\frac{1}{2}})\big)_j,\\[0.3cm]
	&(\mathbf{U}^{\frac{1}{2}}\odot \mathbf{U}^{\frac{1}{2}})_j=U_j^{\frac{1}{2}}U_j^{\frac{1}{2}},\; j=0,1,2,\cdots,N-1.
	\end{align*}
	For better readability, we set
	\begin{align*}
	\begin{array}{ll}
	(F_1)_j^{\frac{1}{2}}=(f^*(u))^{\frac{1}{2}}_j-(f(u))_j^{\frac{1}{2}}, & \quad  (G_1)_j^{\frac{1}{2}}=(g^*(u))^{\frac{1}{2}}_j-(g(u))_j^{\frac{1}{2}},\\[0.3cm]
	(F_2)_j^{\frac{1}{2}}=(f(u))_j^{\frac{1}{2}}-f(u_j^{\frac{1}{2}}),&\quad  (G_2)_j^{\frac{1}{2}}=(g(u))_j^{\frac{1}{2}}-g(u_j^{\frac{1}{2}}),\\[0.3cm]
	(F_3)_j^{\frac{1}{2}}=f(u_j^{\frac{1}{2}})-F(u_j^0,u_j^{1}),&\quad  (G_3)_j^{\frac{1}{2}}=g(u_j^{\frac{1}{2}})-G(u_j^{0},u_j^1),\\[0.3cm]
	(F_4)_j^{\frac{1}{2}}=F(u_j^0,u_j^{1})-F((u^*)_j^0,(u^*)_j^{1}),&\quad (G_4)_j^{\frac{1}{2}}=G(u_j^{0},u_j^1)-G((u^*)_j^{0},(u^*)_j^1),\\[0.3cm]
	(F_5)_j^{\frac{1}{2}}=F((u^*)_j^0,(u^*)_j^{1})-F(U_j^0,U_j^{1}),&\quad (G_5)_j^{\frac{1}{2}}=G((u^*)_j^{0},(u^*)_j^1)-G(U_j^{0},U_j^1).
	\end{array}
	\end{align*}
	According to Lemma \ref{pro-error-estimate}, we have $\|F_1^{\frac{1}{2}}\|_{h}\leq CN^{-r}$ and $\|G_1^{\frac{1}{2}}\|_{h}\leq CN^{-r}$.
	Using Taylor expansion, we get
	$\|F_2^{\frac{1}{2}}\|_{h}\leq C\tau^2$ and $\|G_2^{\frac{1}{2}}\|_{h}\leq C\tau^2.$
	Note that
	\begin{align*}
	&(F_3)_j^{\frac{1}{2}}=\dfrac{\gamma}{3}\left(u_j^{\frac{1}{2}}\partial_xu_j^{\frac{1}{2}}-u_j^{\frac{1}{2}}(\mathbb{D}_1u^{\frac{1}{2}})_j\right)
	=\dfrac{\gamma}{3}\left(u_j^{\frac{1}{2}}\partial_xu_j^{\frac{1}{2}}-u_j^{\frac{1}{2}}\partial_{x}(I_Nu_j^{\frac{1}{2}})\right)
	=\dfrac{\gamma}{3}u_j^{\frac{1}{2}}\left(\partial_x(u_j^{\frac{1}{2}}-I_Nu_j^{\frac{1}{2}})\right),\\[0.3cm]
	&(G_3)_j^{\frac{1}{2}}=\dfrac{\gamma}{3}\left(\partial_x(u_j^{\frac{1}{2}}u_j^{\frac{1}{2}})-\big(\mathbb{D}_1(u^{\frac{1}{2}}\odot u^{\frac{1}{2}})\big)_j\right)=\dfrac{\gamma}{3}\left(\partial_x(u_j^{\frac{1}{2}}u_j^{\frac{1}{2}})-\partial_xI_N(u_j^{\frac{1}{2}}u_j^{\frac{1}{2}})\right)=\dfrac{\gamma}{3}\left(\partial_x\big[u_j^{\frac{1}{2}}u_j^{\frac{1}{2}}-I_N(u_j^{\frac{1}{2}}u_j^{\frac{1}{2}})\big]\right).
	\end{align*}
	It follows from \eqref{semi-norm-projec-eq2} that
	\begin{align*}
	\|F_3^{\frac{1}{2}}\|\leq CN^{-r},\quad \|G_3^{\frac{1}{2}}\|\leq CN^{-r}.
	\end{align*}
	As for $(F_4)_j^{\frac{1}{2}}$ and $(G_4)_j^{\frac{1}{2}}$, we have
	\begin{align*}
	(F_4)_j^{\frac{1}{2}}&=\dfrac{\gamma}{3}\left(u_j^{\frac{1}{2}} (\mathbb{D}_1u)_j^{\frac{1}{2}}-(u^*)_j^{\frac{1}{2}} (\mathbb{D}_1 u^*)_j^{\frac{1}{2}}\right)=\dfrac{\gamma}{3}\left(u_j^{\frac{1}{2}}\partial_xI_Nu_j^{\frac{1}{2}}-(u^*)_j^{\frac{1}{2}}\partial_xI_N(u^*)_j^{\frac{1}{2}}\right)\\[0.3cm]
	&=\dfrac{\gamma}{3}\left([u_j^{\frac{1}{2}}-(u^*)_j^{\frac{1}{2}}]\partial_xI_Nu_j^{\frac{1}{2}}+(u^*)_j^{\frac{1}{2}}[\partial_xI_N(u_j^{\frac{1}{2}}-(u^*)_j^{\frac{1}{2}})]\right),\\[0.3cm]
	(G_4)_j^{\frac{1}{2}}&=\dfrac{\gamma}{3}\left(\big(\mathbb{D}_1(u^{\frac{1}{2}}\odot u^{\frac{1}{2}})\big)_j-\big(\mathbb{D}_1((u^*)^{\frac{1}{2}}\odot (u^*)^{\frac{1}{2}})\big)_j\right)=\dfrac{\gamma}{3}\left(\partial_xI_N(u_j^{\frac{1}{2}}u_j^{\frac{1}{2}})-\partial_xI_N((u^*)_j^{\frac{1}{2}}(u^*)_j^{\frac{1}{2}})\right)\\[0.3cm]
	&=\dfrac{\gamma}{3}\left(   \partial_xI_N[ (u_j^{\frac{1}{2}}-(u^*)_j^{\frac{1}{2}})u_j^{\frac{1}{2}}
	+(u_j^{\frac{1}{2}}-(u^*)_j^{\frac{1}{2}})(u^*)_j^{\frac{1}{2}} ] \right).
	\end{align*}
	Combining the above results with  Lemma \ref{approximation-lemma} and Lemma \ref{semi-pro-error-estimate} leads to
	\begin{align*}
	\|F^{\frac{1}{2}}_4\|_h\leq CN^{-r},\quad \|G^{\frac{1}{2}}_4\|_h\leq CN^{-r}.
	\end{align*}
	We remark that
	\begin{align*}
	(F_5)_j^{\frac{1}{2}}=\dfrac{\gamma}{3}\left(e^{\frac{1}{2}}_j(\mathbb{D}_1(u^*)^{\frac{1}{2}})_j+U_j^{\frac{1}{2}}(\mathbb{D}_1e^{\frac{1}{2}})_j\right)
	=\dfrac{\gamma}{3}\left( e_j^{\frac{1}{2}}\partial_xI_N(u^*)_j^{\frac{1}{2}}+ U_j^{\frac{1}{2}}(\mathbb{D}_1e^{\frac{1}{2}})_j\right).
	\end{align*}
	Using \eqref{semi-norm-equiv-ineq2}, Lemma \ref{approximation-lemma} and Theorem \ref{prior-estimate}, we have
	\begin{align*}
	&\|F_5^{\frac{1}{2}}\|^2_{h}\leq  C(\|e^1\|^2_{h}+|e^1|^2_{h}+\|e^0\|^2_{h}+|e^0|^2_{h}).
	\end{align*}
	Therefore, we can easily deduce
	\begin{align}
	&\|F^{\frac{1}{2}}_{\delta}\|^2_{h}\leq C(\tau^4+N^{-2r})+C(\|e^1\|^2_{h}+|e^{1}|^2_{h}+\|e^0\|^2_{h}+|e^0|^2_{h})
	.\label{F0-delta-eq}
	\end{align}
	As for $(G_5)_j^{\frac{1}{2}}$, we have
	\begin{align*}
	(G_5^{\frac{1}{2}},2e^{\frac{1}{2}})_h& =\bigg(\dfrac{\gamma}{3}\mathbb{D}_1\big((u^*)^{\frac{1}{2}}\odot (u^*)^{\frac{1}{2}}\big) - \dfrac{\gamma}{3}\mathbb{D}_1\big(\mathbf{U}^{\frac{1}{2}}\odot \mathbf{U}^{\frac{1}{2}}\big) ,2e^{\frac{1}{2}}\bigg)_h\\[0.3cm]
	&= \dfrac{\gamma}{3}\bigg(\mathbb{D}_1\big(e^{\frac{1}{2}}\odot ((u^*)^{\frac{1}{2}}+\mathbf{U}^{\frac{1}{2}}) \big),2e^{\frac{1}{2}}\bigg)_h\\[0.3cm]
	&=-\dfrac{2\gamma}{3}\bigg(e^{\frac{1}{2}}\odot \big((u^*)^{\frac{1}{2}}+\mathbf{U}^{\frac{1}{2}}\big),\mathbb{D}_1e^{\frac{1}{2}}\bigg)_h.
	\end{align*}
	Therefore, by Cauchy Schwartz inequality, \eqref{semi-norm-equiv-ineq2} and  Theorem \ref{prior-estimate}, we get
	\begin{align*}
	&|(G_5^{\frac{1}{2}},2e^{\frac{1}{2}})_h|\leq C(\|e^{1}\|_{h}^2+|e^1|_{h}^2+\|e^0\|^2_{h}+|e^0|^2_{h}).
	\end{align*}
	Putting these  results together, we deduce
	\begin{align}
	\begin{split}
	\left|(G_{\delta}^{\frac{1}{2}},2e^{\frac{1}{2}})_h\right|
	&\leq C\big(\|G^{\frac{1}{2}}_1\|_h+\|G^{\frac{1}{2}}_2\|_h+\|G^{\frac{1}{2}}_3\|_h+\|G^{\frac{1}{2}}_4\|_h\big)\|e^{\frac{1}{2}}\|_h+|(G_5^{\frac{1}{2}},2e^{\frac{1}{2}})_h|\\[0.3cm]
	&\leq C(\tau^4+N^{-2r})+C(\|e^1\|^2_{h}+|e^{1}|^2_{h}+\|e^0\|^2_{h}+|e^0|^2_{h}).\label{G0-delta-eq}
	\end{split}
	\end{align}
	Computing the discrete inner product of  \eqref{order2-scheme-error-eq} with  $2e^{\frac{1}{2}}$,  we obtain
	\begin{align*}
	\|e^1\|^2_{0,h}+\sigma|e^1|^2_{h}-(\|e^0\|^2_{h}+\sigma|e^0|^2_{h})=-\tau(F^{\frac{1}{2}}_{\delta},2e^{\frac{1}{2}})_h-\tau(G^{\frac{1}{2}}_{\delta},2e^{\frac{1}{2}})_h-\tau(\xi^0,2e^{\frac{1}{2}})_h.
	\end{align*}
	Using the Cauchy-Schwartz inequality, Lemma \ref{trunc-lem},  \eqref{F0-delta-eq} and \eqref{G0-delta-eq}, we obtain
	\begin{align*}
	\|e^1\|^2_{h}+\sigma|e^1|^2_{h}-(\|e^0\|^2_{h}+\sigma|e^0|^2_{h})\leq C\tau(\tau^4+N^{-2r}+\|e^1\|^2_{h}+\sigma|e^1|^2_{h}+\|e^0\|^2_{h}+\sigma|e^0|^2_{h}).
	\end{align*}	
	When $C\tau\leq 1/2$, we have
	\begin{align}\label{e1-estimate}
	\|e^1\|^2_{h}+\sigma|e^1|^2_{h}\leq
	2C\tau(\tau^4+N^{-2r})+(1+4C\tau)(\|e^0\|^2_{h}+\sigma|e^0|^2_{h}).
	\end{align}
	By Lemma \ref{pro-error-estimate}, Lemma \ref{semi-pro-error-estimate} and noticing $U^0=u^0$, we have
	\begin{align}\label{e0-estimate}
	\begin{split}
	\|e^0\|^2_{h}=\|(u^*)^0-u^0\|^2_{h}\leq CN^{-r},\quad
	|e^0|^2_{h}=|(u^*)^0-u^0|^2_{h}\leq CN^{-r}.
	\end{split}
	\end{align}	
	Substituting \eqref{e0-estimate} into \eqref{e1-estimate} yields
	\begin{align}\label{first-step-error}
	\|e^1\|_{h}+\sigma|e^1|_{h}\leq C(\tau^2+N^{-r}).
	\end{align}	
	The proof is completed.
\end{proof}

\begin{theorem}\label{L2-H1-error-theorem}
	Suppose that $u(x,t)$ is the exact solution of problem \eqref{model-eq} satisfies
	\begin{align*}
	u(x,t)\in C^3(0,T; H^{r+1}_p(\Omega)),\;  r>\dfrac{1}{2},
	\end{align*}
	then the numerical $\mathbf U^n$ of the scheme LCN-MP \eqref{full-discr-LICN} converges to the solution $u(x,t)$ of the problem \eqref{model-eq} without any restrictions in the order of $\mathcal{O}(\tau^2+N^{-r})$ under the discrete $L^{\infty}$ norm.
\end{theorem}
\begin{proof}
	Let $e_j^n=(u^*)^n-U_j^n$.
	Subtracting \eqref{full-discr-LICN}  from \eqref{trunc-eq}  leads to the following error equation
	\begin{align}\label{error-eq-F-G} \xi^n = (\mathbb{I}-\sigma\mathbb{D}_2) \delta^+_te^n + a\mathbb{D}_1e^{n+\frac{1}{2}} + (F_{\delta})^{n+\frac{1}{2}} + (G_{\delta})^{n+\frac{1}{2}}, ~n\geq 1,
	\end{align}
	where
	\begin{align*}
	\begin{split}
	&(F_{\delta})_j^{n+\frac{1}{2}}=(f^*(u))^{n+\frac{1}{2}}_j-F(U_j^{n-1},U_j^n,U_j^{n+1}),\quad F(U_j^{n-1},U_j^n,U_j^{n+1})=\dfrac{\gamma}{3}\widehat{U}_j^{n+\frac{1}{2}}(\mathbb{D}_1\mathbf{U}^{n+\frac{1}{2}})_j,\\[0.3cm]
	&(G_{\delta})_j^{n+\frac{1}{2}}=(g^*(u))^{n+\frac{1}{2}}_j-G(U_j^{n-1},U_j^n,U_j^{n+1}),\quad G(U_j^{n-1},U_j^n,U_j^{n+1})=\dfrac{\gamma}{3}(\mathbb{D}_1(\widehat{\mathbf{U}}^{n+\frac{1}{2}}\odot \mathbf{U}^{n+\frac{1}{2}}))_j.
	\end{split}
	\end{align*}
	For simplicity, let
	\begin{align*}
	&(F_1)_j^{n+\frac{1}{2}}=(f^*(u))^{n+\frac{1}{2}}_j-(f(u))_j^{n+\frac{1}{2}},\\[0.3cm] &(F_2)_j^{n+\frac{1}{2}}=(f(u))_j^{n+\frac{1}{2}}-f(u_j^{n+\frac{1}{2}}),\\[0.3cm]
	&(F_3)_j^{n+\frac{1}{2}}=f(u_j^{n+\frac{1}{2}})-F(u_j^{n-1},u_j^n,u_j^{n+1}),\\[0.3cm]
	&(F_4)_j^{n+\frac{1}{2}}=F(u_j^{n-1},u_j^n,u_j^{n+1})-F((u^*)_j^{n-1},(u^*)_j^n,(u^*)_j^{n+1}),\\[0.3cm]
	&(F_5)_j^{n+\frac{1}{2}}=F((u^*)_j^{n-1},(u^*)_j^n,(u^*)_j^{n+1})-F(U_j^{n-1},U_j^n,U_j^{n+1}),
	\end{align*}
	and
	\begin{align*}
	&(G_1)_j^{n+\frac{1}{2}}=(g^*(u))^{n+\frac{1}{2}}_j-(g(u))_j^{n+\frac{1}{2}},\\[0.3cm]
	&(G_2)_j^{n+\frac{1}{2}}=(g(u))_j^{n+\frac{1}{2}}-g(u_j^{n+\frac{1}{2}}),\\[0.3cm]
	&(G_3)_j^{n+\frac{1}{2}}=g(u_j^{n+\frac{1}{2}})-G(u_j^{n-1},u_j^n,k_j^{n+1}),\\[0.3cm]
	&(G_4)_j^{n+\frac{1}{2}}=G(u_j^{n-1},u_j^n,u_j^{n+1})-G((u^*)_j^{n-1},(u^*)_j^n,(u^*)_j^{n+1}),\\[0.3cm]
	&(G_5)_j^n=G((u^*)_j^{n-1},(u^*)_j^n,(u^*)_j^{n+1})-G(U_j^{n-1},U_j^n,U_j^{n+1}).
	\end{align*}
	Similar to the proof of Lemma \ref{start-step-lem}, we can obtain
	\begin{align*}
	&\|F_1^{n+\frac{1}{2}}\|_{h}\leq CN^{-r},\quad
	\|F_2^{n+\frac{1}{2}}\|_{h}\leq C\tau^2,\quad
	\|F^{n+\frac{1}{2}}_3\|_{h}\leq CN^{-r},\quad
	\|F^{n+\frac{1}{2}}_4\|_h\leq CN^{-r},\quad\\[0.3cm]
	&\|G_1^{n+\frac{1}{2}}\|_{h}\leq CN^{-r},\quad
	\|G_2^{n+\frac{1}{2}}\|_{h}\leq C\tau^2,\quad	
	\|G^{n+\frac{1}{2}}_3\|_{h}\leq CN^{-r},\quad
	\|G^{n+\frac{1}{2}}_4\|_h\leq CN^{-r}.
	\end{align*}
	In the following, we estimate $(F_5)_j^{n+\frac{1}{2}}$ and $(G_5)_j^{n+\frac{1}{2}}$ one by one.
	On the one hand,
	\begin{align*}
	(F_5)_j^{n+\frac{1}{2}}&=\dfrac{\gamma}{3}\left(\hat{e}^{n+\frac{1}{2}}_j(\mathbb{D}_1(u^*)^{n+\frac{1}{2}})_j+\widehat{U}_j^{n+\frac{1}{2}}(\mathbb{D}_1e^{n+\frac{1}{2}})_j\right) \\[0.3cm]
	&=\dfrac{\gamma}{3}\left( \hat{e}^{n+\frac{1}{2}}_j\partial_xI_N(u^*)_j^{n+\frac{1}{2}}+ \widehat{U}_j^{n+\frac{1}{2}}(\mathbb{D}_1e^{n+\frac{1}{2}})_j\right).
	\end{align*}
	Using \eqref{semi-norm-equiv-ineq2}, Theorem \ref{prior-estimate} and Lemma \ref{approximation-lemma}, we have
	\begin{align*}
	\|F_5^{n+\frac{1}{2}}\|^2_{h}
	\leq C(\|e^n\|^2_{h}+\|e^{n-1}\|^2_{h}+|e^{n+1}|^2_{h}+|e^n|^2_{h}).
	\end{align*}
	Based on the above results, we can deduce
	\begin{align}
	\|F^{n+\frac{1}{2}}_{\delta}\|^2_{h}
	\leq C(\tau^4+N^{-2r})+C(\|e^n\|^2_{h}+\|e^{n-1}\|^2_{h}+|e^{n+1}|^2_{h}+|e^n|^2_{h}).\label{F-delta-eq}
	\end{align}
	On the other hand,
	\begin{align*}
	(G_5^{n+\frac{1}{2}},2e^{n+\frac{1}{2}})_h& = \dfrac{\gamma}{3}\bigg(\mathbb{D}_1\big(\hat{e}^{n+\frac{1}{2}}\odot(u^*)^{n+\frac{1}{2}}\big) + \mathbb{D}_1(\widehat{\mathbf{U}}^{n+\frac{1}{2}}\odot e^{n+\frac{1}{2}}),2e^{n+\frac{1}{2}}\bigg)_h\\[0.3cm]
	&=-\dfrac{2\gamma}{3}\bigg(\hat{e}^{n+\frac{1}{2}}\odot(u^*)^{n+\frac{1}{2}},\mathbb{D}_1e^{n+\frac{1}{2}}\bigg)_h-\dfrac{2\gamma}{3}\bigg(\widehat{\mathbf{U}}^{n+\frac{1}{2}}\odot e^{n+\frac{1}{2}},\mathbb{D}_1e^{n+\frac{1}{2}}\bigg)_h.
	\end{align*}
	Thus, by Cauchy Schwartz inequality, Theorem \ref{prior-estimate} and \eqref{semi-norm-equiv-ineq2}, we reach
	\begin{align*}
	|(G_5^{n+\frac{1}{2}},2e^{n+\frac{1}{2}})_h|
	\leq C(\|e^{n-1}\|^2_{h}+\|e^{n}\|^2_{h}+\|e^{n+1}\|^2_{h}+|e^{n}|^2_{h} +|e^{n+1}|^2_{h}  ).
	\end{align*}
	As a result,
	\begin{align}
	\begin{split}
	\left|(G_{\delta}^{n+\frac{1}{2}},2e^{n+\frac{1}{2}})_h\right|&\leq C\big(\|G^{n+\frac{1}{2}}_1\|_h+\|G^{n+\frac{1}{2}}_2\|_h+\|G^{n+\frac{1}{2}}_3\|_h+\|G^{n+\frac{1}{2}}_4\|_h\big)\|e^{n+\frac{1}{2}}\|_h+|(G_5^{n+\frac{1}{2}},2e^{n+\frac{1}{2}})_h|\\[0.3cm]
	&\leq C(\tau^4+N^{-2r})+C(\|e^{n-1}\|^2_{h}+\|e^{n}\|^2_{h}+\|e^{n+1}\|^2_{h}+|e^{n}|^2_{h} +|e^{n+1}|^2_{h}).\label{G-delta-inner-prodc}
	\end{split}
	\end{align}
	Computing the discrete inner product of \eqref{error-eq-F-G} with $2e^{n+\frac{1}{2}}$, we obtain
	\begin{align}\label{nlevel-inner-prod-eq}
	\|e^{n+1}\|^2_{h}+\sigma|e^{n+1}|^2_{h}-(\|e^{n}\|^2_{h}+\sigma|e^{n}|^2_{h})=\tau(\xi^n,2e^{n+\frac{1}{2}})_h-\tau(F^{n+\frac{1}{2}}_{\delta},2e^{n+\frac{1}{2}})_h-\tau(G^{n+\frac{1}{2}}_{\delta},2e^{n+\frac{1}{2}})_h.
	\end{align}
	For each term in the right-hand of \eqref{nlevel-inner-prod-eq}, using Lemma \ref{trunc-lem}, \eqref{F-delta-eq}, \eqref{G-delta-inner-prodc} and  Cauchy-Schwartz inequality yields
	\begin{align}\label{some-NL-inner-produ}
	\begin{split}
	&|\tau(\xi^n,2e^{n+\frac{1}{2}})_h|\leq C\tau(\tau^4+\|e^n\|^2_{h}+\|e^{n+1}\|^2_{h}),\\[0.3cm]
	&|\tau(F_{\delta}^{n+\frac{1}{2}},2e^{n+\frac{1}{2}})_h|\leq C\tau(\tau^4+N^{-2r}+\|e^{n+1}\|^2_{h}+\|e^n\|^2_{h}+\|e^{n-1}\|^2_{h}+|e^{n+1}|^2_{h}+|e^n|^2_{h}),\\[0.3cm]
	&|\tau(G_{\delta}^{n+\frac{1}{2}},2e^{n+\frac{1}{2}})_h|\leq C\tau(\tau^4+N^{-2r}+\|e^{n-1}\|^2_{h}+\|e^{n}\|^2_{h}+\|e^{n+1}\|^2_{h}+|e^{n}|^2_{h} +|e^{n+1}|^2_{h}).
	\end{split}
	\end{align}
	Substituting \eqref{some-NL-inner-produ} into \eqref{nlevel-inner-prod-eq} gives
	\begin{align*}
	\begin{split}
	&\|e^{n+1}\|^2_{h}+\sigma|e^{n+1}|^2_{h}-(\|e^{n}\|^2_{h}+\sigma|e^{n}|^2_{h})\\[0.3cm]
	&\leq C\tau(\tau^4+N^{-2r})+C\tau(\|e^{n-1}\|^2_{h}+\|e^n\|_h^2+\|e^{n+1}\|_h^2+|e^n|_h^2+|e^{n+1}|^2_{h})\\[0.3cm]
	&\leq C\tau(\tau^4+N^{-2r})+C\tau(\|e^{n-1}\|^2_{h}+\|e^n\|_h^2+\|e^{n+1}\|_h^2+\sigma|e^n|_h^2+\sigma|e^{n+1}|^2_{h}).
	\end{split}
	\end{align*}
	Replacing $n$ by $l$ and summing the above equation together for $l$ from $1$ to $n$, we arrive at
	\begin{align}\label{sum-inner-result}
	\begin{split}
	&\|e^{n+1}\|^2_{h}+\sigma|e^{n+1}|^2_{h}-(\|e^{1}\|^2_{h}+\sigma|e^{1}|^2_{h})
	\leq C(\tau^4+N^{-2r})+C\tau\sum_{l=0}^{n+1}(\|e^{l}\|_h^2+\sigma|e^{l}|_h^2),
	\end{split}
	\end{align}
	where we have noted that $n\tau\leq T$.
	Substituting \eqref{first-step-error}  into  \eqref{sum-inner-result}, and when $C\tau\leq 1/2$, we get
	\begin{align*}
	\|e^{n+1}\|^2_{h}+\sigma|e^{n+1}|^2_{h}\leq C(\tau^4+N^{-2r})+C\tau\sum_{l=1}^{n}(\|e^{l+1}\|^2_{h}+\sigma|e^{l+1}|^2_{h}).
	\end{align*}
	By Lemma \ref{Discr-Sobol-ineq} (discrete Gronwall inequality), we have
	\begin{align*}
	\|e^{n+1}\|^2_{h}+\sigma|e^{n+1}|^2_{h}\leq C(\tau^4+N^{-2r}),
	\end{align*}
	which implies
	\begin{align}\label{2-norm-semi-norm-result}
	\|e^{n+1}\|_{h}+\sigma|e^{n+1}|_{h}\leq C(\tau^2+N^{-r}).
	\end{align}
	It follows from Lemma \ref{pro-error-estimate}, Lemma \ref{semi-pro-error-estimate} and \eqref{2-norm-semi-norm-result} that
	\begin{align}
	&\|u^n-\mathbf{U}^n\|_h\leq \|u^n-(u^*)^n\|_h+\|(u^*)^n-\mathbf{U}^n\|_h\leq C(\tau^2+N^{-r}),\label{L2-estimate}\\[0.3cm]
	&|u^n-\mathbf{U}^n|_h\leq |u^n-(u^*)^n|_h+|(u^*)^n-\mathbf{U}^n|_h\leq C(\tau^2+N^{-r}).\label{H1-estimate}
	\end{align}	
	By \eqref{semi-norm-equiv-ineq1} and \eqref{H1-estimate}, we have
	\begin{align}\label{diff-quoti-estimate}
	\|\delta_x^+(u^n-\mathbf{U}^n)\|_h\leq C(\tau^2+N^{-r}).
	\end{align}	
	Hence, from \eqref{L2-estimate},  \eqref{diff-quoti-estimate} and Lemma \ref{Discr-Sobol-ineq}, it follows
	\begin{align*}
	\|u^n-\mathbf{U}^n\|_{\infty,h}\leq C(\tau^2+N^{-r}).
	\end{align*}
	This completes the proof.
\end{proof}

\begin{theorem}
	Suppose that $u(x,t)$ is the exact solution of problem \eqref{model-eq} satisfies
	\begin{align*}
	u(x,t)\in C^3(0,T; H^{r+1}_p(\Omega)),\; r>\dfrac{1}{2},
	\end{align*}
	then the numerical $\mathbf U^n$ of the scheme LLF-MP \eqref{full-discr-LILF} converges to the solution $u(x,t)$ of the problem \eqref{model-eq} without any restrictions in the order of $\mathcal{O}(\tau^2+N^{-r})$ under the discrete $L^{\infty}$ norm.
\end{theorem}	
\begin{proof}
	The proof is similar to that of Theorem \ref{L2-H1-error-theorem} and is thus omitted.
\end{proof}

\section{Numerical experiments}\label{Numer-exper}
In this section, some numerical experiments are carried out to show the performance of the schemes LCN-MP and LLF-MP.
The performance of proposed methods will be showed in following aspects:
\begin{itemize}
	\item
	to test the accuracy order of the schemes LCN-MP and LLF-MP.
	\item
	to simulate the migration of the solitary waves.
	\item
	to show the performance in preserving the momentum property.
	\item
	to make comparison with some existing methods.
\end{itemize}
To quantify the numerical results, we define the discrete $L^2$ error and the discrete $L^{\infty}$ error at $t=t_n$ as
\begin{align*}
\|E_u\|^2_{0,h}=h\sum_{j=0}^{N-1}|u(x_j,t_n)_j-u_j^n|^2,\quad
\|E_u\|^2_{\infty,h}=\max_{0\leq j\leq N-1}|u(x_j,t_n)-u_j^n|.
\end{align*}
The corresponding rates of convergence both in time and space are obtained by the formula below
\begin{align*}
\mathrm{Order}=\dfrac{\log(error_1/error_2)}{\log(\delta_1/\delta_2)},
\end{align*}
where $\delta_j,\ error_j\ (j=1,2)$ are step size and the corresponding error with step size $\delta_j$, respectively.
In order to show the preservation of invariants at $n$-th time level, the relative mass,  momentum and energy error at $t=t_n$ are respectively defined as
$$
RI_1=\dfrac{|{\mathcal{I}^n_1}_h-{\mathcal{I}^0_1}_h|}{|{\mathcal{I}^0_1}_h|},\quad
RI_2=\dfrac{|{\mathcal{I}^n_2}_h-{\mathcal{I}^0_2}_h|}{|{\mathcal{I}^0_2}_h|},\quad
RI_3=\dfrac{|{\mathcal{I}^n_3}_h-{\mathcal{I}^0_3}_h|}{|{\mathcal{I}^0_3}_h|}.
$$
where
${\mathcal{I}^n_2}_h=h\sum_{j=0}^{N-1}U_j^n$,\;
${\mathcal{I}^n_2}_h=h\sum_{j=0}^{N-1}((U_j^n)^2-\sigma u_j^n(\mathbb{D}_2\mathbf{U})_j)$ and ${\mathcal{I}^n_3}_h=h\sum_{j=0}^{N-1}(\frac{\gamma}{6}(U^n_j)^3+\frac{a}{2}(U_j^n)^2)$ are the discrete mass, momentum and energy, respectively.
Moreover,
some schemes involved in this section are
given in Table \ref{Table:notations}.
\begin{table}[!htbp]
	\caption{The notations for the various schemes used in the numerical computation.\label{Table:notations}}
	\begin{center}
		\begin{tabular}{lllll}\hline\specialrule{0.0em}{2.0pt}{2.0pt}
			&Notation   &Algorithm description \\\specialrule{0em}{2.0pt}{2.0pt}\hline\specialrule{0.0em}{2.0pt}{2.0pt}
			& LCN-MP   &The algorithm is defined in \eqref{full-discr-LICN}.\\
			& LLF-MP  &The algorithm is defined in \eqref{full-discr-LILF}.\\
			&ELMP-I &The scheme comes from \cite{Caij2017}.\\
			&ELMP-II& The scheme  comes from \cite{Caij2017}.\\
			&ILMP-I &The scheme comes from \cite{Caij2017}.\\
			&ILMP-II&The scheme comes from \cite{Caij2017}.\\
			\hline\specialrule{0em}{2.5pt}{2.5pt}
		\end{tabular}
	\end{center}
\end{table}

\subsection{Migration of a single solitary wave}
The RLW equation has an analytic solution of the form
\begin{align*}
u(x,t)=3c\mathrm{sech}^2(k[x-x_0-vt]),\quad
k=\dfrac{1}{2}\sqrt{\dfrac{\gamma c}{\sigma(a+\gamma c)}},
\end{align*}
which corresponds to the motion of a single solitary wave with amplitude $3c$, initial center at $x_0$, the wave velocity $v=a+\gamma c$. All computations are done with  $a=\gamma=\sigma=1$, $x_0=0$.

\subsubsection{Test accuracy in space and in time}
To investigate the accuracy in space, we take $\tau=1.0e-4$ so that the error in the temporal direction can be negligible. With grid sizes from $N=32$ to $N=64$ in increment of $4$, we solve \eqref{model-eq} by LCN-MP and LLF-MP up to time $T=1$.
For exploring the time accuracy, we fix the space step $N=1024$, so that the numerical errors are dominated mainly by the temporal ones. With a sequence of time step $\tau=0.1,0.05,0.025,0.0125,0.00625$, we also compute the numerical errors at $T=1$. In the two cases, we choose $c=3/2$ and set the space interval $x\in[-30,30]$.
The errors of the numerical solution in discrete $L^2$ and $L^{\infty}$ norm  are presented in Fig. \ref{fig:space-accuracy} and Fig. \ref{fig:time-accuracy},  where   a second-order accuracy in time and spectral accuracy in space  are shown clearly. The accuracy test validates the correctness of our methods.

\begin{figure}[!htbp]
	\centering
	\subfigure[LCN-MP]{
		\includegraphics[width=0.35\textwidth,height=0.35\textwidth]{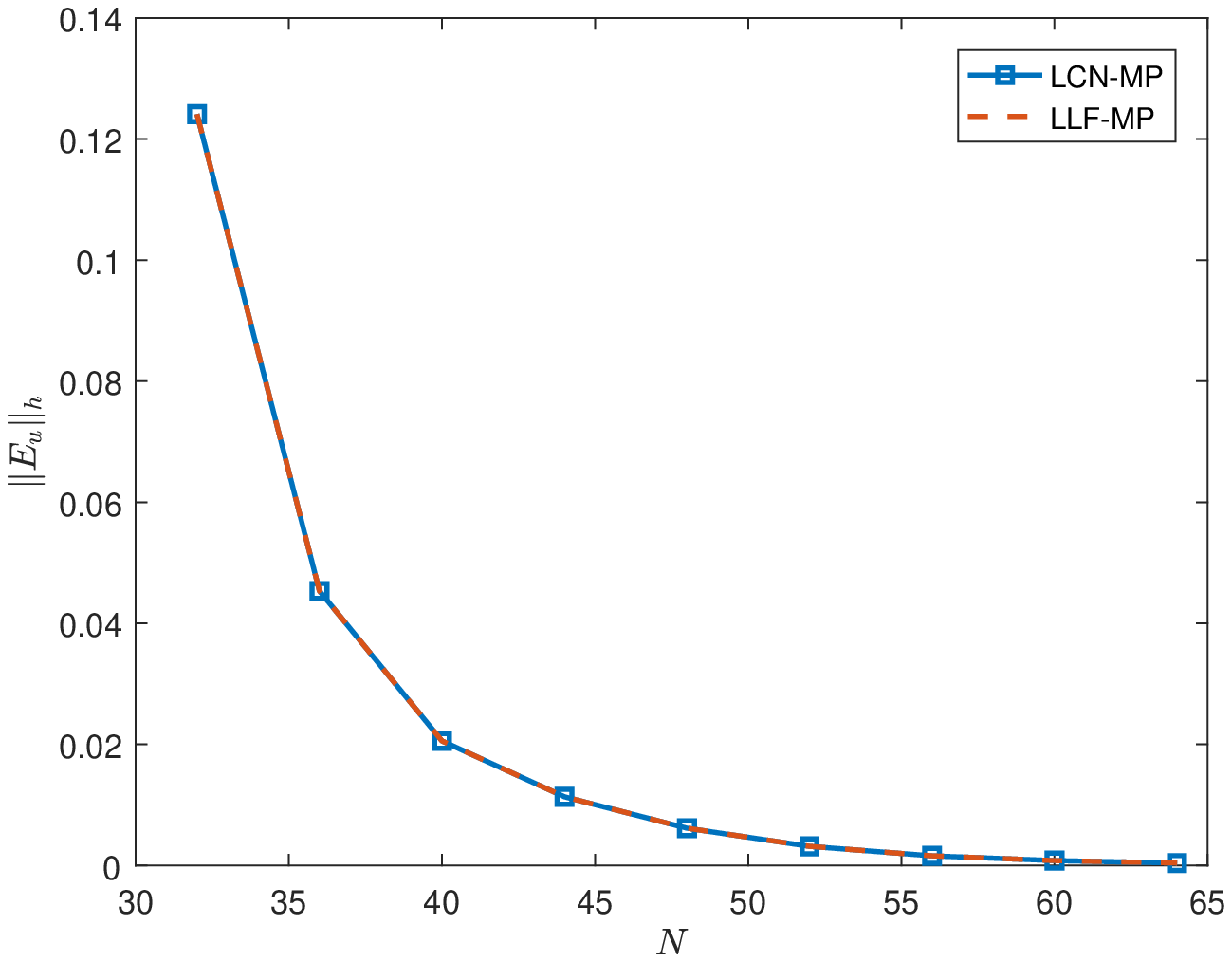}}\quad
	\subfigure[LLF-MP]{
		\includegraphics[width=0.35\textwidth,height=0.35\textwidth]{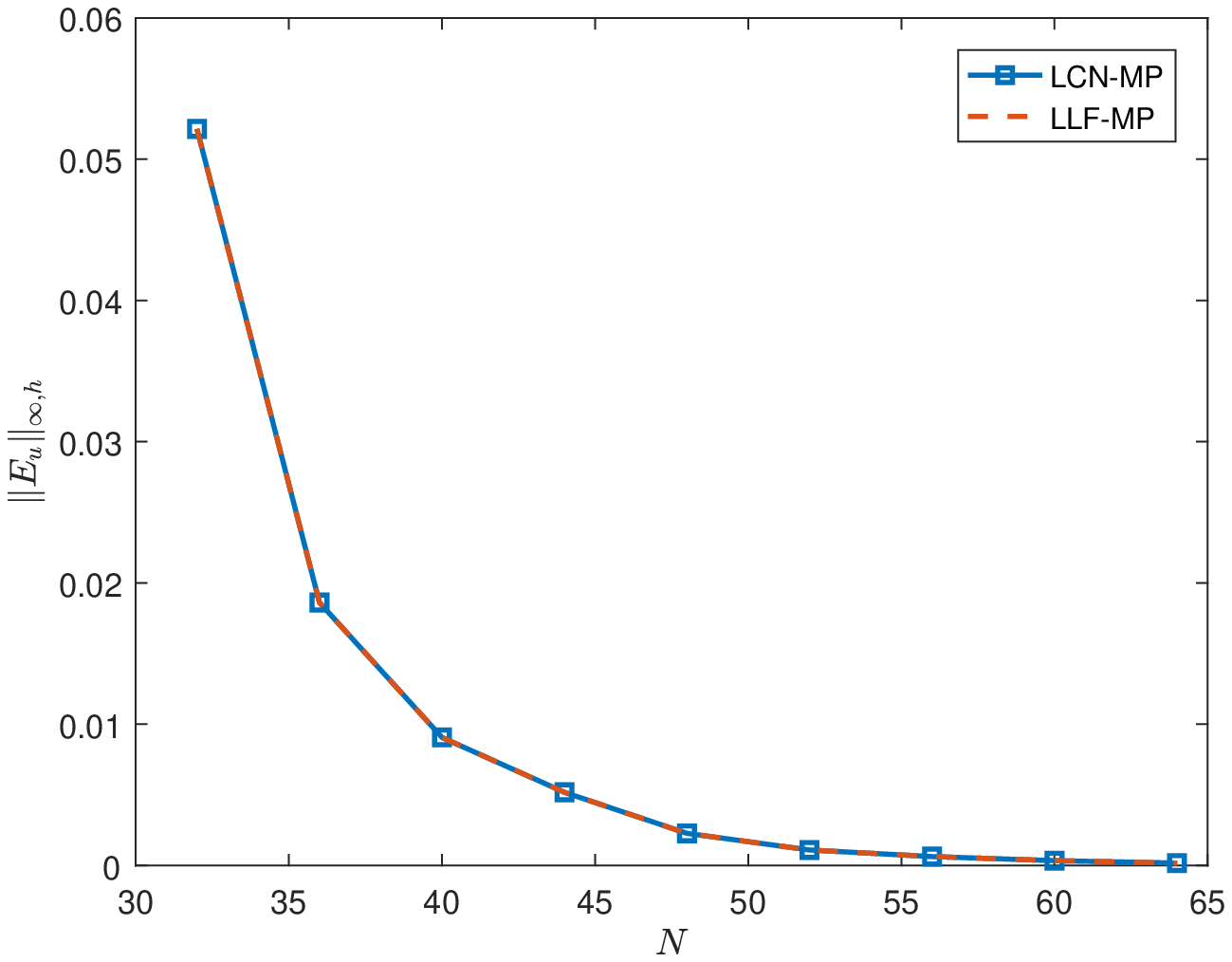}}
	\caption{\small  The accuracy of numerical for the space direction with using fixed time step $\tau=1.0e-4$ by the schemes LCN-MP and LLF-MP.  \label{fig:space-accuracy}}
\end{figure}

\begin{figure}[!htbp]
	\centering
	\subfigure[LCN-MP]{
		\includegraphics[width=0.35\textwidth,height=0.35\textwidth]{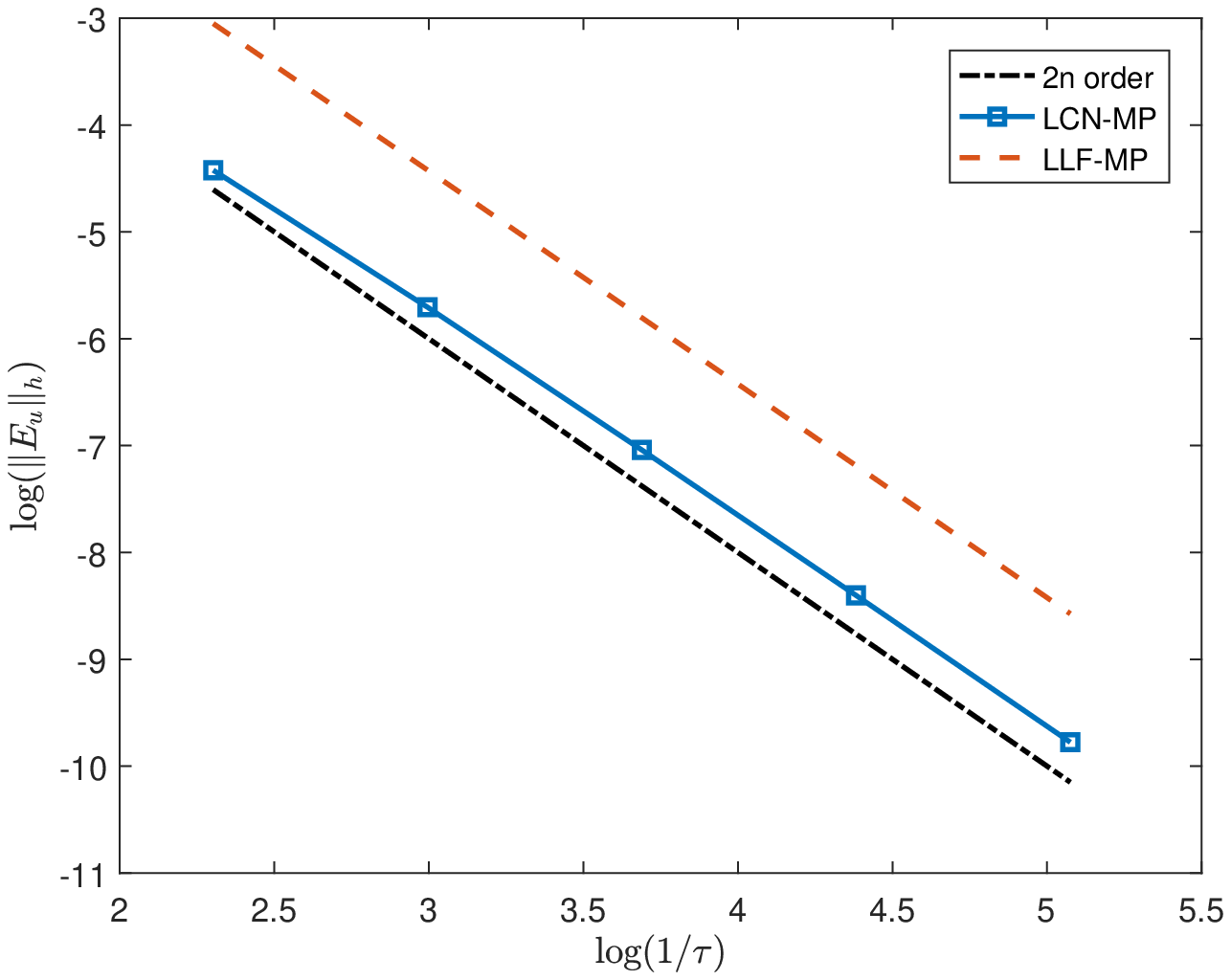}}\quad
	\subfigure[LLF-MP]{
		\includegraphics[width=0.35\textwidth,height=0.35\textwidth]{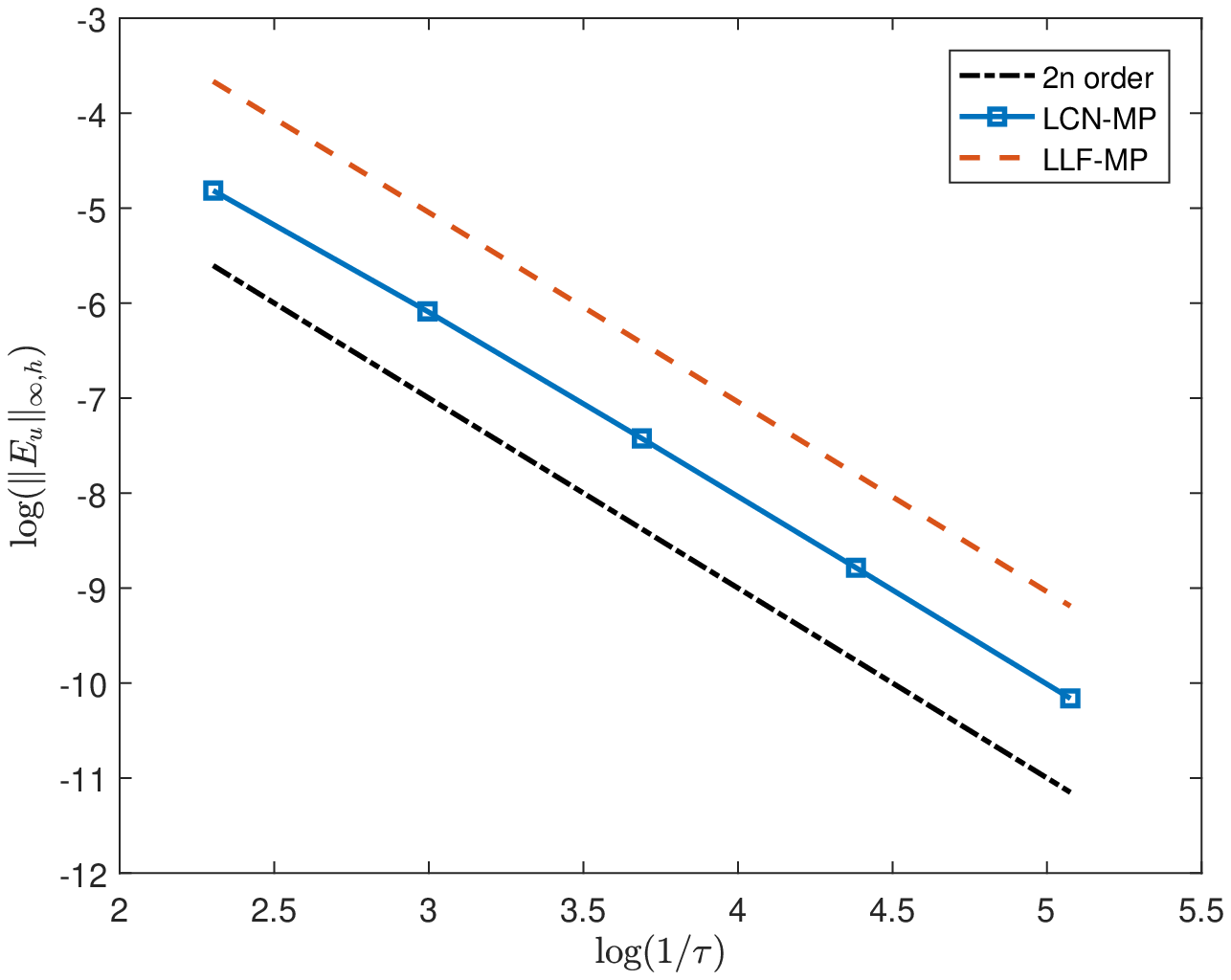}}
	\caption{\small   The accuracy of numerical for the temporal direction with using fixed space step $N=1024$ by the schemes LCN-MP and LLF-MP.  \label{fig:time-accuracy}}
\end{figure}

\subsubsection{The propagation of a single solitary wave}
In this test, the proposed schemes are performed with $x\in [-30,30]$, $N=256$, $\tau=1.0e-3$ and $c=1/3$.
Fig. \ref{fig:travel-fig} presents the wave profile of the numerical solution for RLW equation from $t=0$ to $t=6$.
Compared with the exact wave profile, we can see clearly that the wave shapes of LCN-MP and LLF-MP  are captured very well.

\begin{figure}[!htbp]
	\centering
	\subfigure[Exact solution]{
		\includegraphics[width=0.3\textwidth,height=0.3\textwidth]{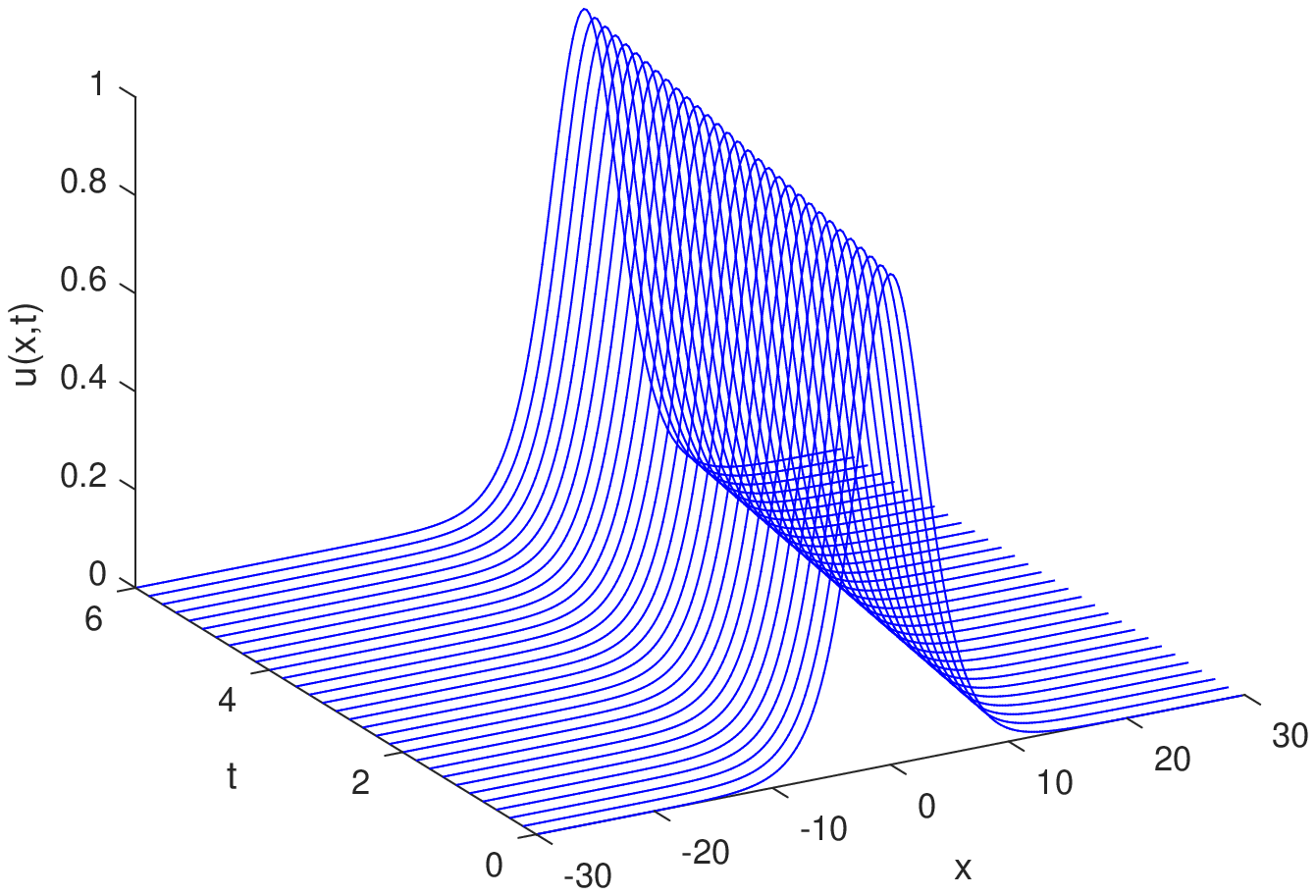}}
	\subfigure[LCN-MP]{
		\includegraphics[width=0.3\textwidth,height=0.3\textwidth]{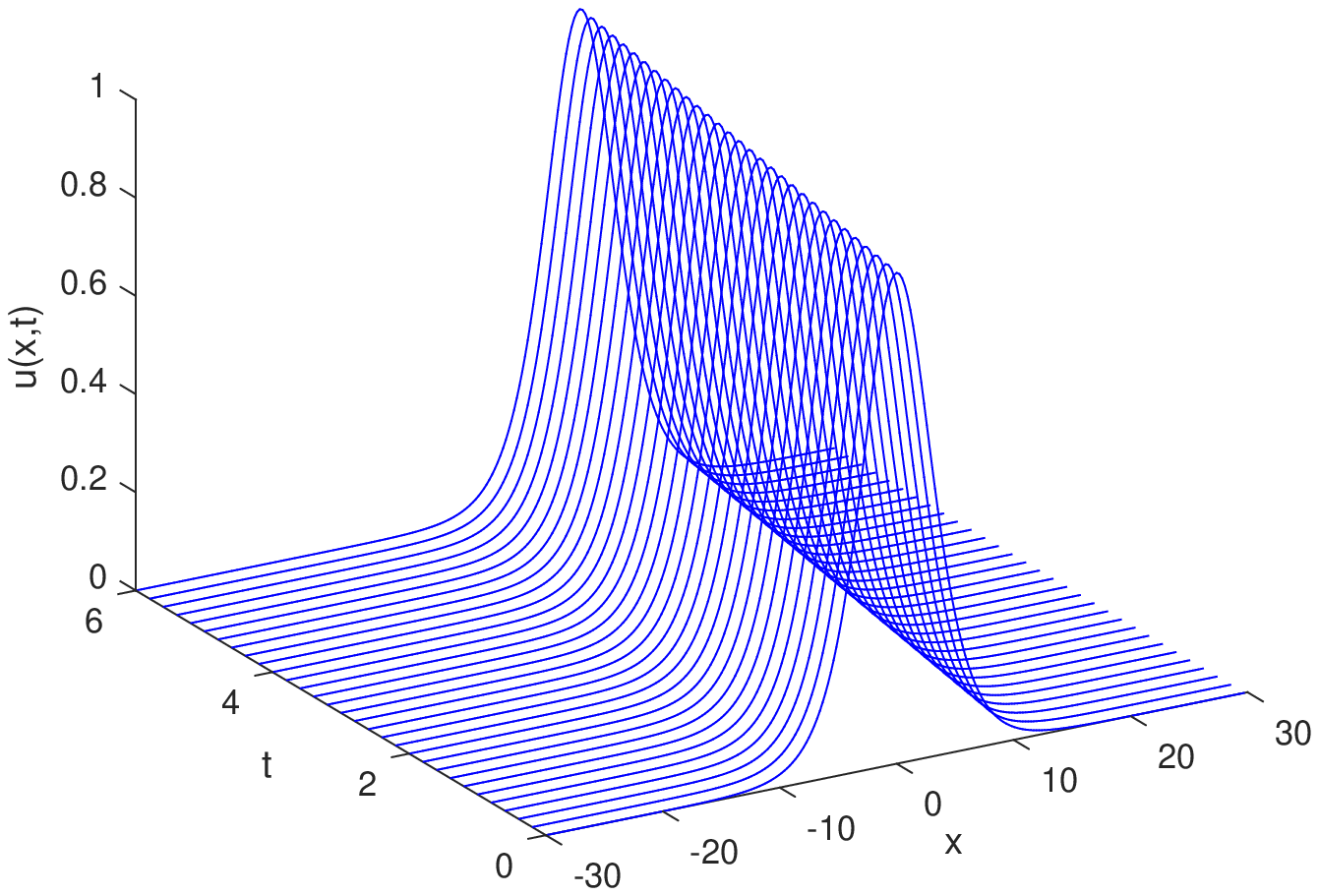}}
	\subfigure[LLF-MP]{
		\includegraphics[width=0.3\textwidth,height=0.3\textwidth]{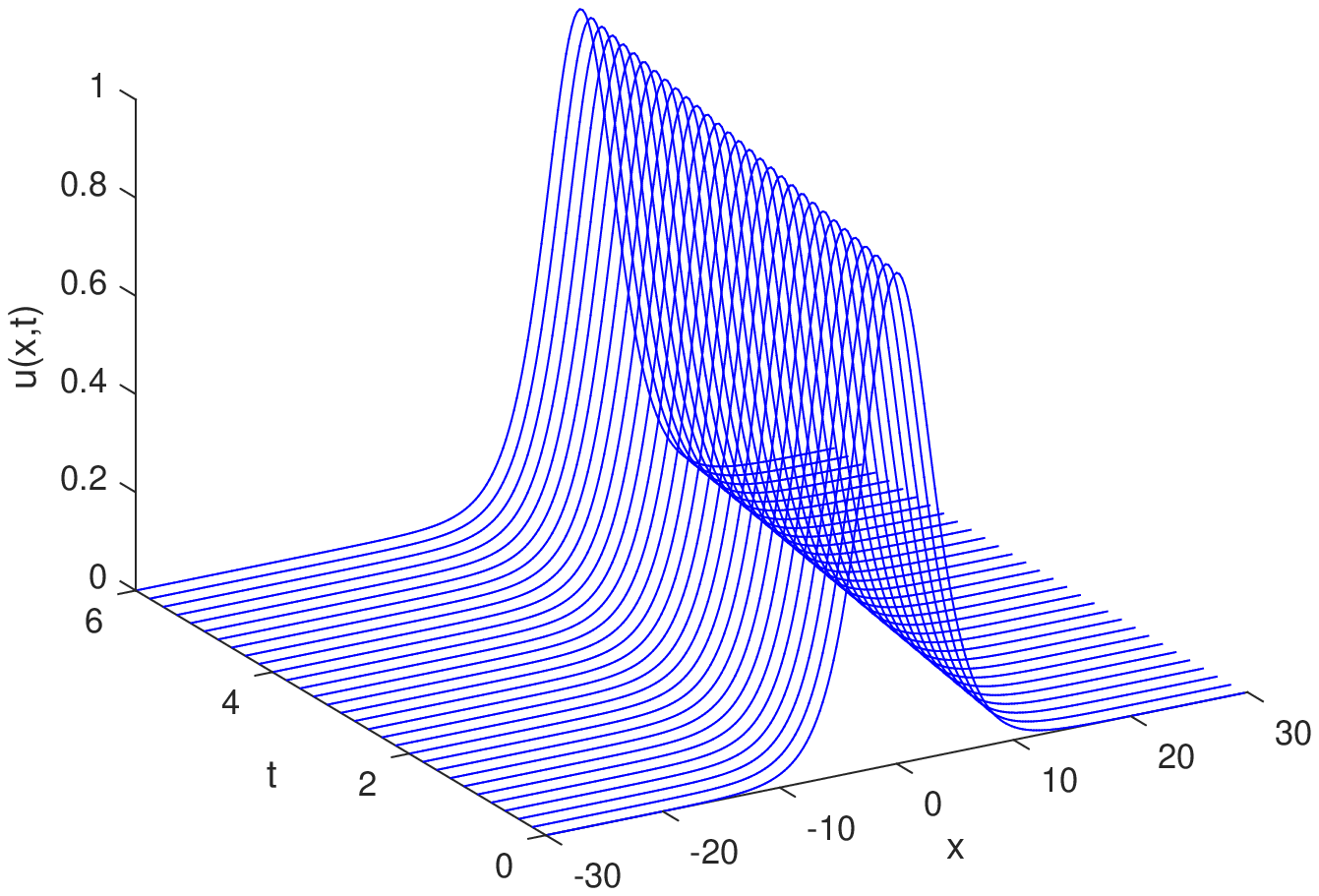}}
	\caption[]{\small Profile of a single solitary wave $u(x,t)$ from $t=0$ to $t=6$.   \label{fig:travel-fig}}
\end{figure}

\subsubsection{Test conservation properties}
In order to test these conservation properties, we take   $\tau=0.025$ with $N=256$, $c=1/3$ and computational interval $x\in[-30,30]$. The run of the algorithm is continued up to  $T=100$. In view of the relative errors in the  mass, momentum and energy conservation laws (see Fig. \ref{fig:test-invariants}), we can find that the discrete momentum can be preserved to round-off errors by the schemes LCN-MP and LLF-MP. In addition, the two schemes
do not preserve the mass and energy exactly. The conclusion is consistent with our theoretical result.
\begin{figure}[!htbp]
	\centering
	\subfigure[]{
		\includegraphics[width=0.30\textwidth,height=0.30\textwidth]{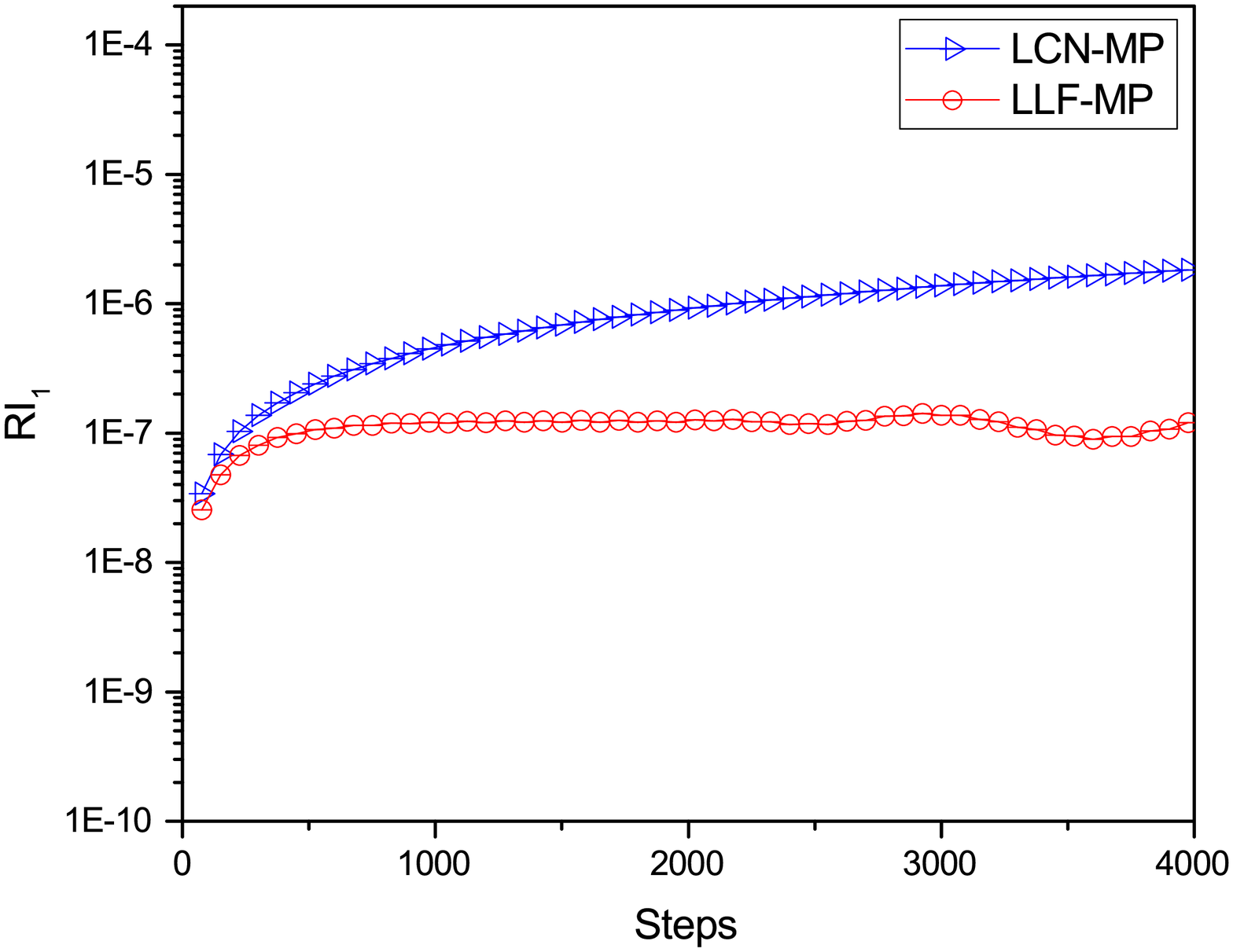}}
	\subfigure[]{
		\includegraphics[width=0.30\textwidth,height=0.30\textwidth]{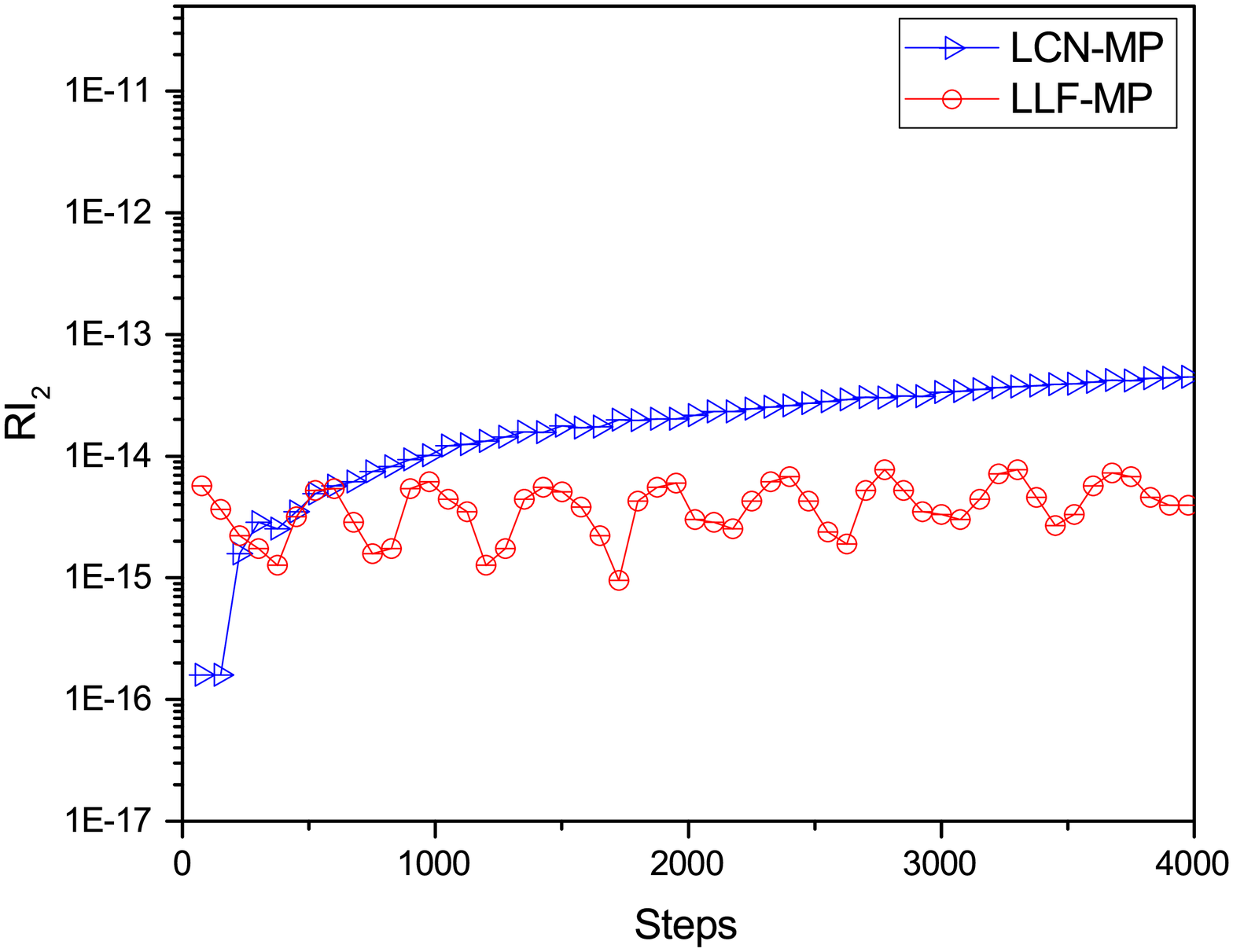}}
	\subfigure[]{
		\includegraphics[width=0.30\textwidth,height=0.30\textwidth]{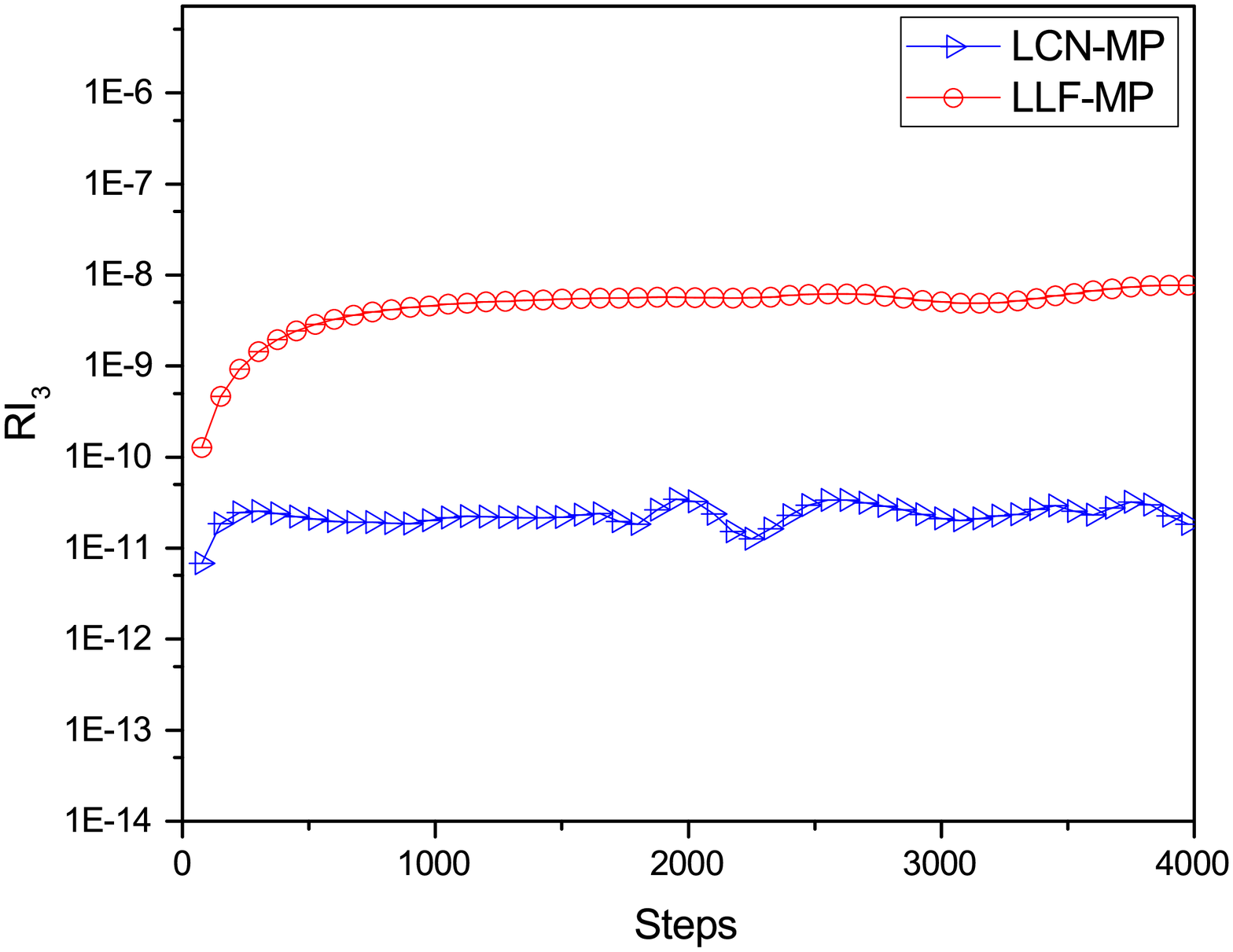}}
	\caption{\small Relative mass error $RI_1$ (a), momentum error $RI_2$ (b), and energy error $RI_3$ (c) of the schemes LCN-MP and LLF-MP for a single solitary wave.   \label{fig:test-invariants}}
\end{figure}

\subsubsection{Compared with some existing schemes}
To show  advantages of our proposed schemes, we compare them with some existing schemes. We choose  the computational domain $x\in[-60,200]$. The run of all algorithm is continued up to time $T = 75$ with  $\tau=0.05$ and $h=0.1$.

The $L^{\infty}$ error and $L^2$ error in solution at different times are displayed in Table \ref{tab:Numer-error-comparison}. Compared with four existing schemes (ELMP-I, ELMP-II, ILMP-I and ILMP-II), one can see that the schemes LCN-MP and LLF-MP perform satisfactory solutions in long-time computation and  LCN-MP provides the most accurate solution than the others.
Table \ref{tab:Numer-Compari} displays the errors in solution and the CPU times for the schemes ILMP-I, ILMP-II, LCN-MP and LLF-MP. It is clear that the errors of the four schemes  decrease as $c$ decreases. Moreover, the schemes LCN-MP and LLF-MP  admit not only  much smaller error but also are more efficient than the rest of ones.  The reason is that  a linear system needs to be solved at each time step, which highly improves the efficiency of numerical computation.
The relative errors of discrete momentum are produced by six different momentum-preserving methods in Fig. \ref{fig:compared-invariants}, but the schemes LCN-MP and LLF-MP are superior than  ILMP-I and ILMP-II.

\begin{table}[!htbp]
	{\caption{The error norms in solution for the single solitary wave  with $\tau=0.05$, $h=0.1$ and $-60\leq x\leq 200$.}\label{tab:Numer-error-comparison}}
	\begin{center}
		\begin{tabular}{l l l l l l l l}\hline\specialrule{0em}{1pt}{1pt}
			\multirow{2}{*}{Method} &\multirow{2}{*}{Error} &\multicolumn{3}{c}{c=1/3} &\multicolumn{3}{c}{c=1/2}\\ \cmidrule(lr) {3-5}  \cmidrule(lr){6-8}
			\quad &\quad &T=25  &T=50 & T=75   &T=25  &T=50 & T=75 \\ \hline\specialrule{0.0em}{2.0pt}{2.0pt}
			ELMP-I & $L^2$ & 3.02e-3 &4.51e-3 &5.85e-3   &6.44e-3 &9.83e-3 &1.33e-2\\
			\quad  & $L^{\infty}$&1.27e-3 &1.83e-3 &2.35e-3    &2.85e-3 &4.26e-3 &5.67e-3 \\
			\specialrule{0.0em}{2.0pt}{2.0pt}
			ELMP-II &$L^2$ &2.14e-3 &3.70e-3 &5.19e-3 &3.61e-3 &6.69e-3 &9.76e-3\\
			\quad  & $L^{\infty}$&8.67e-4 &1.44e-3 &2.00e-3 &1.48e-3 &2.71e-3 &3.94e-3\\
			\specialrule{0.0em}{2.0pt}{2.0pt}
			ILMP-I & $L^2$ &2.49e-4&3.50e-4&4.71e-4&1.37e-3&2.89e-3&4.42e-3 \\
			\quad  & $L^{\infty}$&6.64e-5&1.12e-4&1.67e-4&5.58e-4&1.17e-3&1.79e-3 \\
			\specialrule{0.0em}{2.0pt}{2.0pt}
			ILMP-II &$L^2$ &5.00e-3&8.18e-3&1.12e-2&1.11e-2&1.91e-2&2.72e-2\\
			\quad  & $L^{\infty}$&2.10e-3&3.28e-3&4.43e-3&4.84e-3&8.08e-3&1.13e-2\\
			\specialrule{0.0em}{2.0pt}{2.0pt}		
			LCN-MP&$L^2$&2.20e-4&4.06e-4&5.84e-4&3.41e-4&5.96e-4&8.80e-4 \\
			\quad  & $L^{\infty}$&9.42e-5&1.62e-4&2.28e-4&1.69e-4&2.86e-4&4.05e-4\\	
			\specialrule{0.0em}{2.0pt}{2.0pt}	
			LLF-MP&$L^2$&3.28e-3&5.59e-3&7.82e-3 &8.86e-3&1.59e-2&2.30e-2\\
			\quad  & $L^{\infty}$&1.36e-3&2.21e-3&3.05e-3 &3.79e-3&6.62e-3&9.46e-3\\		
			\hline
		\end{tabular}
	\end{center}
\end{table}

\begin{table}[!htbp]
	{\caption{Numerical comparison by various methods at $T=100$ with $\tau=0.05$, $h=0.1$ and $-60\leq x\leq 200$.}\label{tab:Numer-Compari}}
	\begin{center}
		\begin{tabular}{l l l l l l l}\hline\specialrule{0em}{1pt}{1pt}
			\multirow{2}{*}{Method} &\multicolumn{3}{c}{$c=0.1$} &\multicolumn{3}{c}{$c=0.03$}\\ \cmidrule(lr) {2-4} \cmidrule(lr){5-7}
			&$L^2$ error       &$L^{\infty}$ error &CPU(s)  &$L^2$ error       &$L^{\infty}$ error  &CPU(s) \\ \hline\specialrule{0em}{1pt}{1pt}
			ILMP-I & 3.13e-4 &1.01e-4&35.57 &4.61e-5&1.16e-5&30.22\\
			ILMP-II& 1.40e-3 &4.61e-4&34.14 &1.39e-4&3.98e-5&29.02\\
			LCN-MP&1.24e-4&3.99e-5&5.29 &2.59e-5&5.08e-6&5.25   \\
			LLF-MP&7.05e-4&2.31e-4&5.76 &6.65e-5&1.84e-5&4.73 \\
			\hline
		\end{tabular}
	\end{center}
\end{table}

\begin{figure}[!htbp]
	\centering
	\subfigure{
		\includegraphics[width=0.75\textwidth,height=0.55\textwidth]{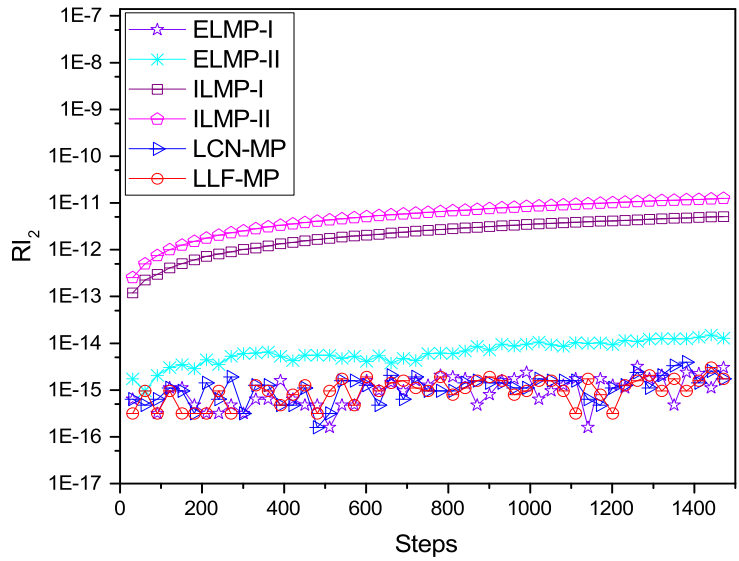}}
	\caption{\small The relative momentum errors of six momentum-preserving methods with $c=1/3$ until $T=75$.    \label{fig:compared-invariants}}
\end{figure}

\subsection{Interaction of two positive solitary waves}
In this test, we study the interaction of two positive solitary waves having different amplitudes and traveling in the same direction. We consider the initial conditions
\begin{align*}
u(x,0)=3c_1\mathrm{sech}^2(m_1(x-x_1))+3c_2\mathrm{sech}^2(m_2(x-x_2)),
\end{align*}
where $m_1=\frac{1}{2}\sqrt{\frac{\gamma c_1}{(\gamma c_1+1)\delta}}$, $m_2=\frac{1}{2}\sqrt{\frac{\gamma c_2}{(\gamma c_2+1)\delta}}$, $x_1=-20$, $x_2=15$, $c_1=1$, $c_2=0.5$. The analytical momentum value can be found as
\begin{align}\label{exact-invariants-eq}
I_2=\sum_{i=1}^2\left(\dfrac{6c^2_i}{m_i}+\dfrac{24m_ic_i^2\delta}{5}\right).
\end{align}
The simulation is performed with   $\tau=0.05$, $h=0.1$,  $\gamma=1$, $\sigma=1$ and $-60\leq x \leq 280.$
\begin{figure}[!htbp]
	\centering
	\subfigure{
		\includegraphics[width=0.35\textwidth,height=0.30\textwidth]{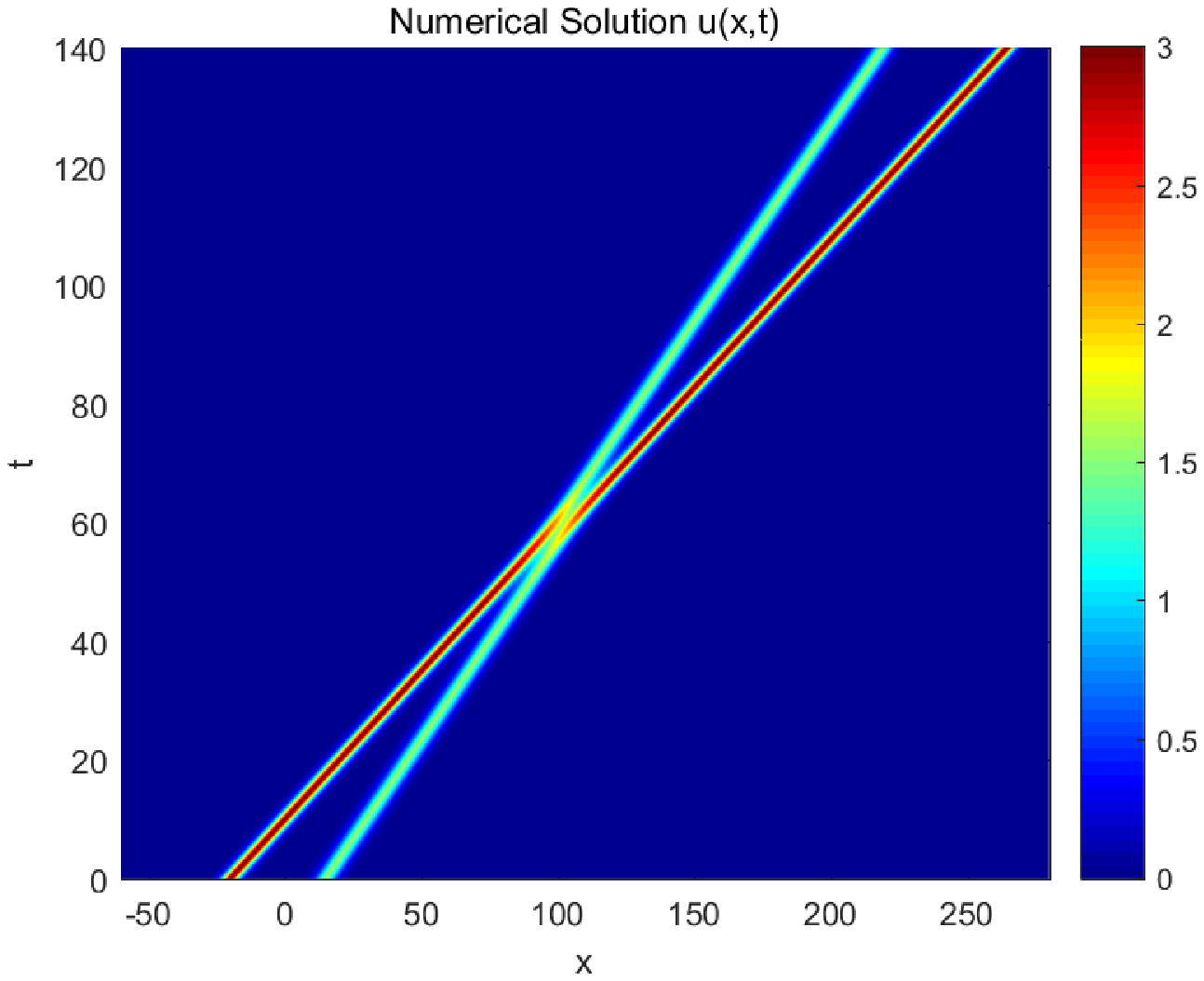}}
	\subfigure{
		\includegraphics[width=0.35\textwidth,height=0.30\textwidth]{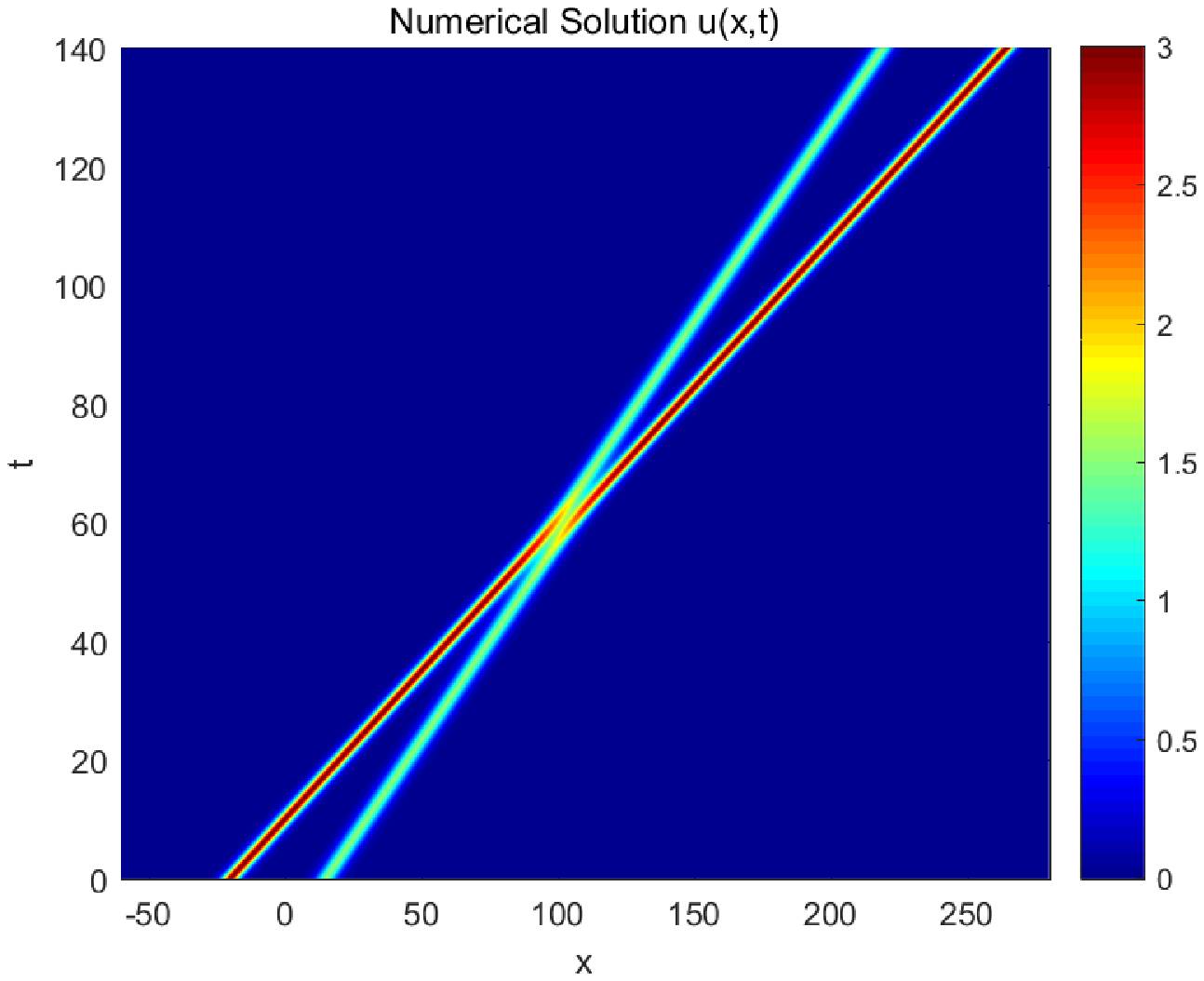}}
	\caption{\small . The profile of numerical solution computed by the schemes LCN-MP (left) and LLF-MP (right).\label{fig:two-solitary}}
\end{figure}
Fig. \ref{fig:two-solitary} shows the profile of numerical solution of the interaction of two positive solitary waves from $t=0$ to $t=140$.  In Table \ref{tab:two-solitary-momentum}, the numerical results of the momentum invariant are obtained by different methods. One can see that the invariant of  momentum by the schemes LCN-MP and LLF-MP almost coincide with analytical values throughout.
The changes in momentum for the four schemes are displayed in Fig \ref{fig:two-solitray-error-momentum}. The results imply that the momentum is captured exactly throughout the simulation.

\begin{table}[!htbp]
	\begin{center}
		\caption{The momentum value for interaction of two solitary waves with different numerical methods.\label{tab:two-solitary-momentum}}\vspace{0.3cm}
		\begin{tabular}{l cc cc cc cc}
			\hline\specialrule{0em}{1pt}{1pt}
			\multirow{2}{*}{Time}&
			\multicolumn{1}{c}{Analytical value}&
			\multicolumn{1}{c}{ELMP-I} &
			\multicolumn{1}{c}{ELMP-II} &
			\multicolumn{1}{c}{ILMP-I} &
			\multicolumn{1}{c}{ILMP-II} &
			\multicolumn{1}{c}{LCN-MP} &
			\multicolumn{1}{c}{LLF-MP} &\\
			\cmidrule(lr){2-2} \cmidrule(lr){3-3} \cmidrule(lr){4-4} \cmidrule(lr){5-5} \cmidrule(lr){6-6} \cmidrule(lr){7-7} \cmidrule(lr){8-8}
			& $I_2$ & ${\mathcal{I}_2}_h$ & ${\mathcal{I}_2}_h$ & ${\mathcal{I}_2}_h$ & ${\mathcal{I}_2}_h$ & ${\mathcal{I}_2}_h$ & ${\mathcal{I}_2}_h$\\
			\midrule
			t=0  &24.210182&24.191390&24.196489 &24.204501 &24.209608 &24.210182  &24.210182 \\
			t=20 &24.210182&24.191390&24.196489  &24.204501 &24.209608 &24.210182  &24.210182  \\
			t=40 &24.210182&24.191390&24.196489  &24.204501 &24.209608 &24.210182  &24.210182  \\
			t=60 &24.210182&24.191390&24.196489  &24.204501 &24.209608 &24.210182  &24.210182 \\
			t=80 &24.210182&24.191390&24.196489  &24.204501 &24.209608 &24.210182  &24.210182 \\
			t=100&24.210182&24.191390&24.196489  &24.204501 &24.209608 &24.210182  &24.210182 \\
			\hline
		\end{tabular}
	\end{center}
\end{table}

\begin{figure}[!htbp]
	\centering
	\subfigure{
		\includegraphics[width=0.3\textwidth,height=0.3\textwidth]{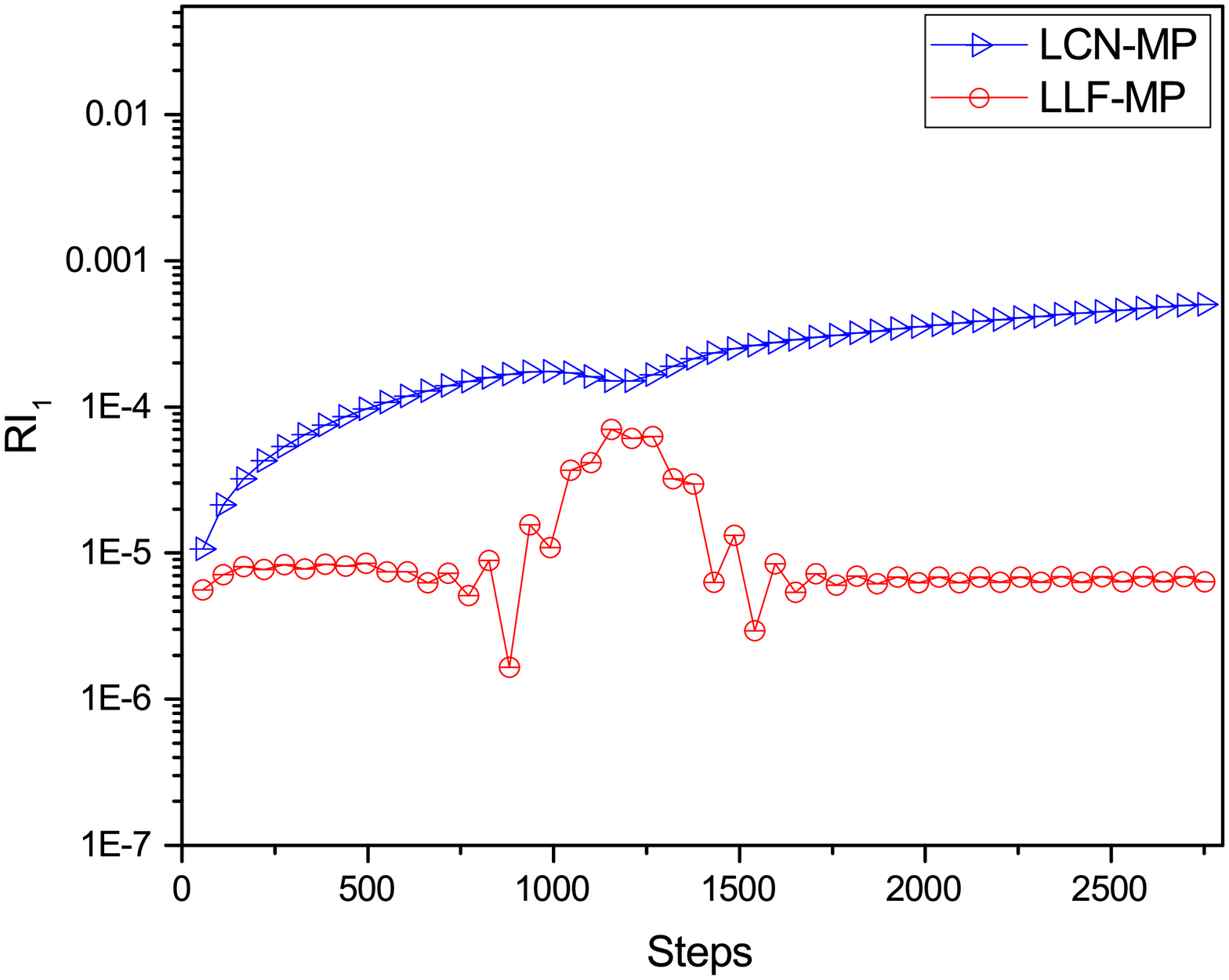}}
	\subfigure{
		\includegraphics[width=0.3\textwidth,height=0.3\textwidth]{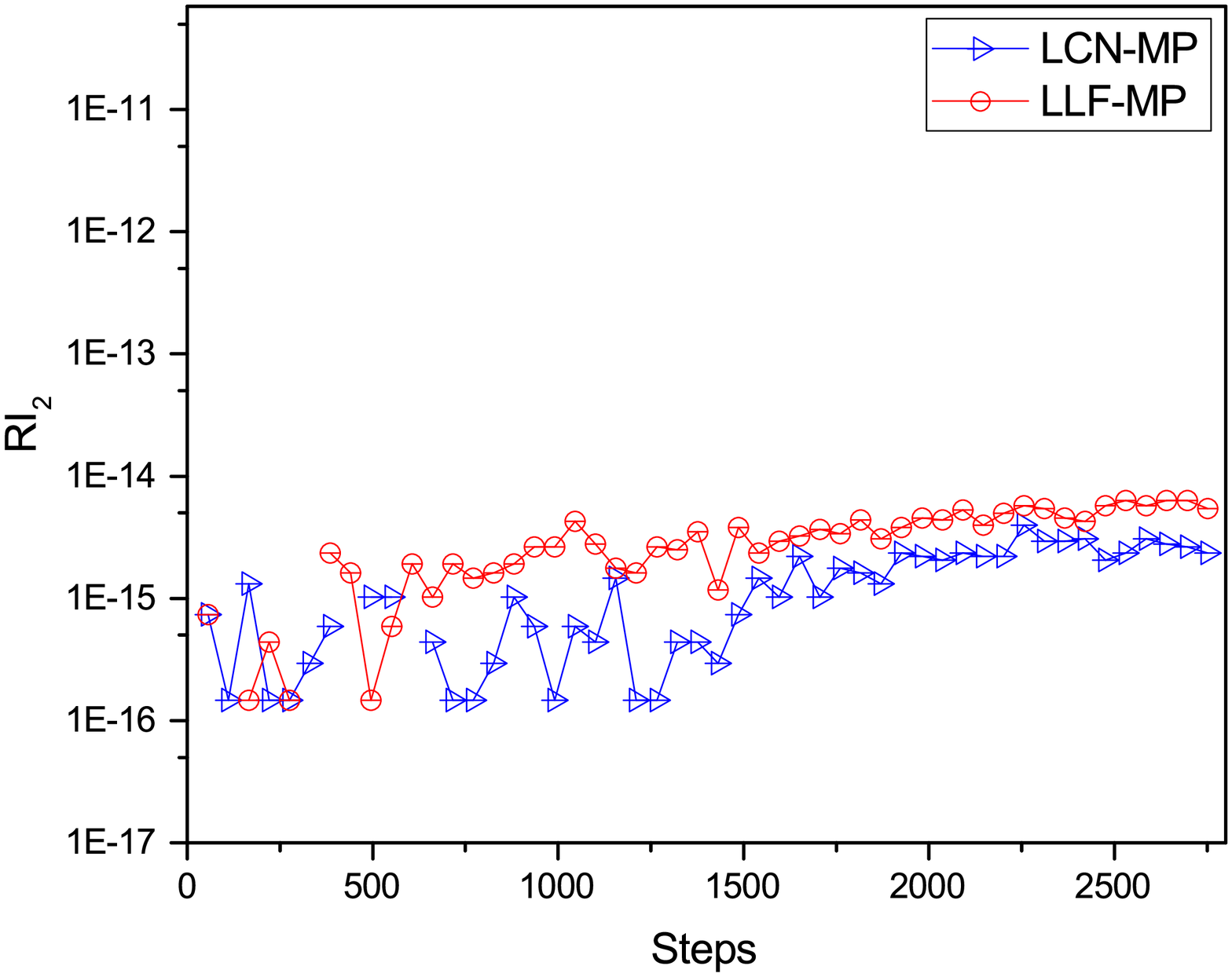}}
	\subfigure{
		\includegraphics[width=0.3\textwidth,height=0.3\textwidth]{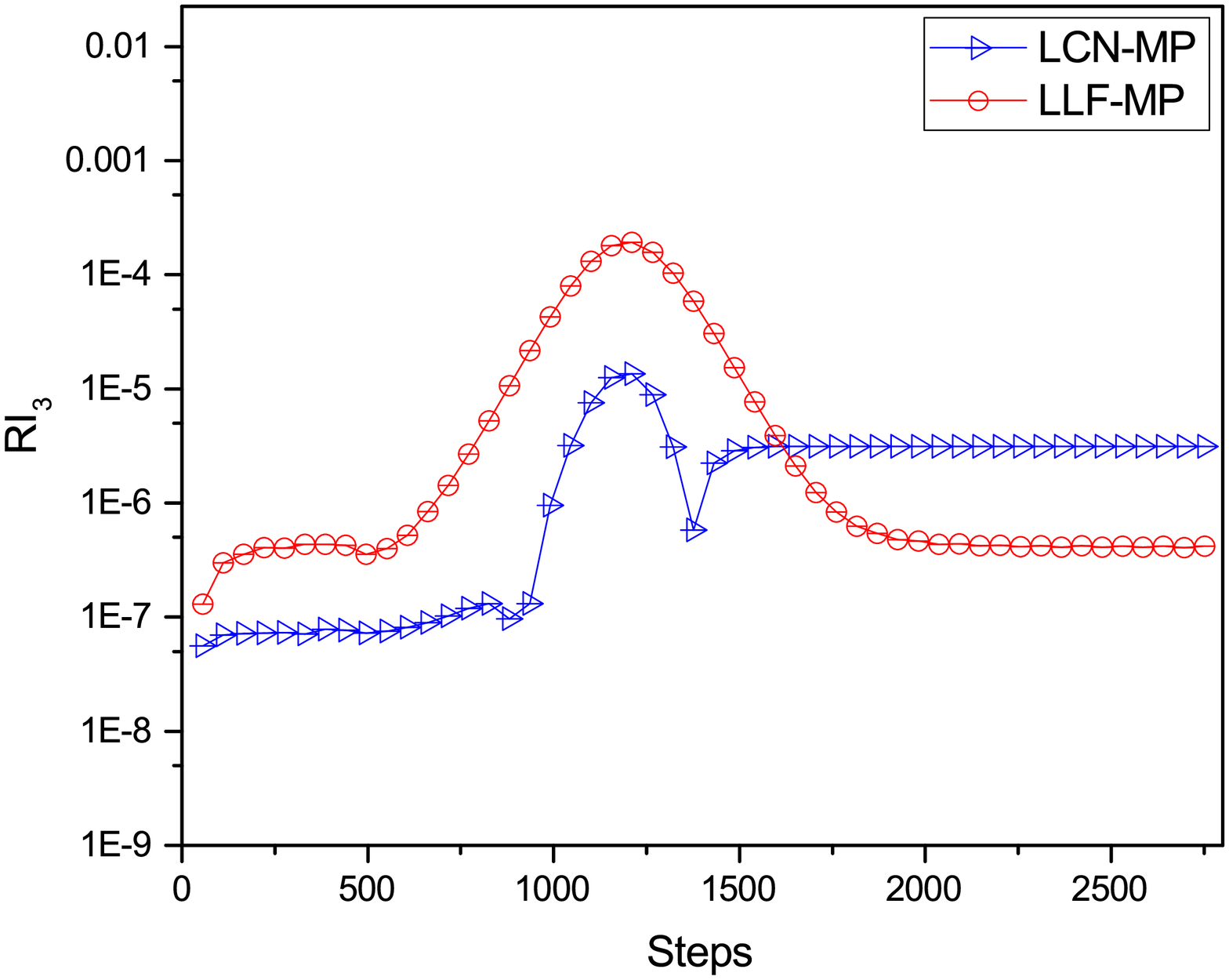}}
	\caption{\small The relative errors in mass, momentum and energy of the schemes LCN-MP and LLF-MP with $c=1$, $\tau=0.05$, $h=0.1$ and $x\in[-60,280]$ until $T=140$.\label{fig:two-solitray-error-momentum}}
\end{figure}

\subsection{The Maxwellian pulse}
In this part, we have examined the evolution of an initial Maxwellian pulse into solitary waves for various values of the parameter $\sigma$.
We take the initial condition
\begin{align*}
u(x,0)=\exp(-(x-7)^2),\quad -40\leq x\leq 100,
\end{align*}
and all simulations are done with $\gamma=1$, $a=1$, $\tau=0.01$ and $h=0.1$.

We  discuss each of the following cases: (i) $\sigma=0.04$, (ii) $\sigma=0.01$ and (iii) $\sigma=0.001$, respectively. The simulation starts at $T=0$ and stops at $T=40$.
Fig. \ref{fig:maxwellian-motion} shows that more and more solitary waves are formed with reducing the value of $\sigma$ by the schemes LCN-MP and LLF-MP.  We can find that only a single soliton is generated for $\sigma=0.04$,  while for $\sigma=0.01$ three stable solitons are generated.
For $\sigma = 0.001$, the Maxwellian pulse decays into about eight solitary waves.
The relative changes in momentum for $\sigma=0.04$, $\sigma=0.01$ and $\sigma=0.001$ are respectively displayed in Fig. \ref{fig:maxwellian-momentum-0001}. It is clear that the schemes LCN-MP and LLF-MP both capture the momentum well and the former performs better than the latter as $\sigma$ decreases.

\begin{figure}[!htbp]
	\centering
	\subfigure{
		\includegraphics[width=0.320\textwidth,height=0.30\textwidth]{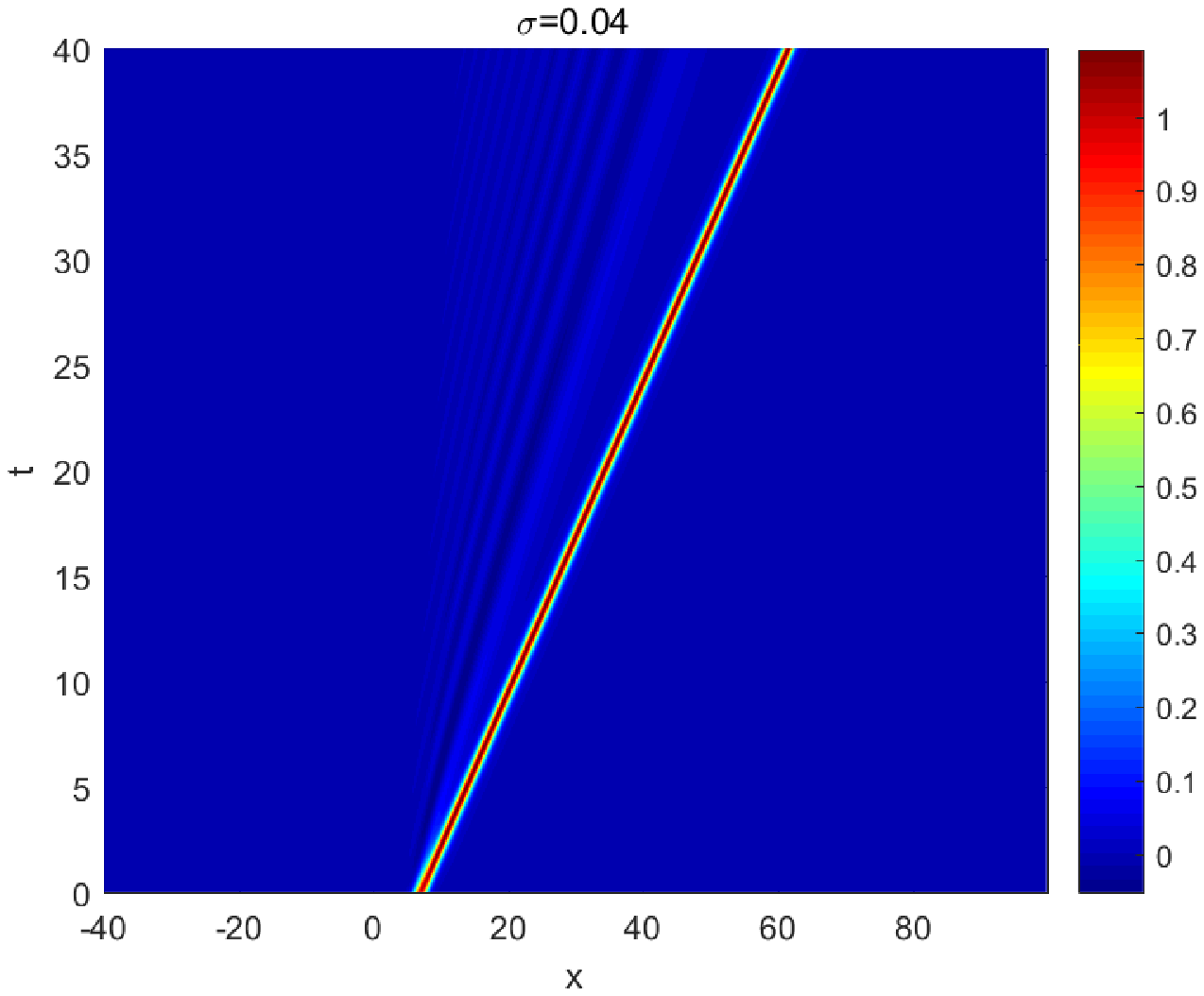}}
	\subfigure{
		\includegraphics[width=0.320\textwidth,height=0.30\textwidth]{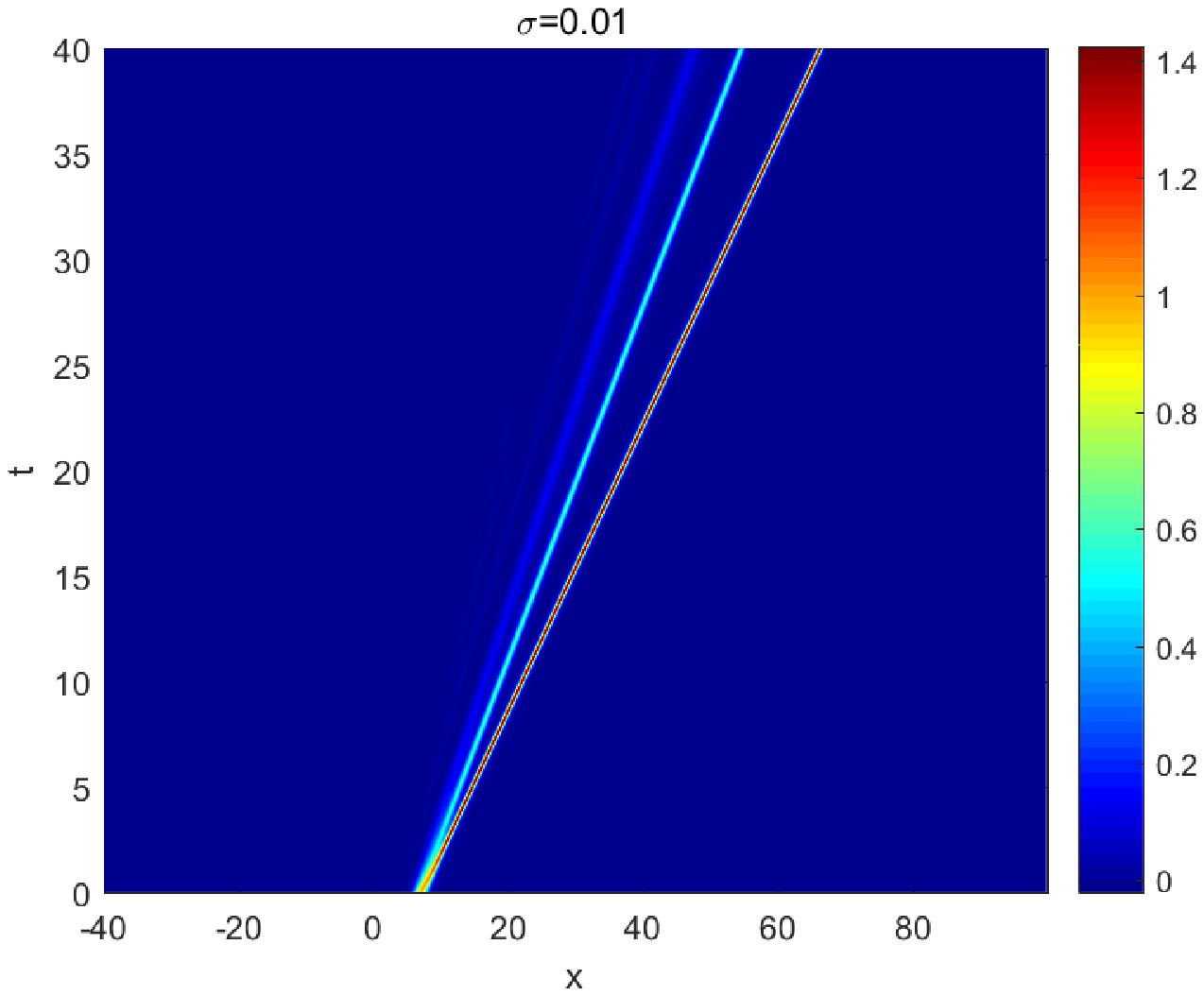}}
	\subfigure{
		\includegraphics[width=0.320\textwidth,height=0.30\textwidth]{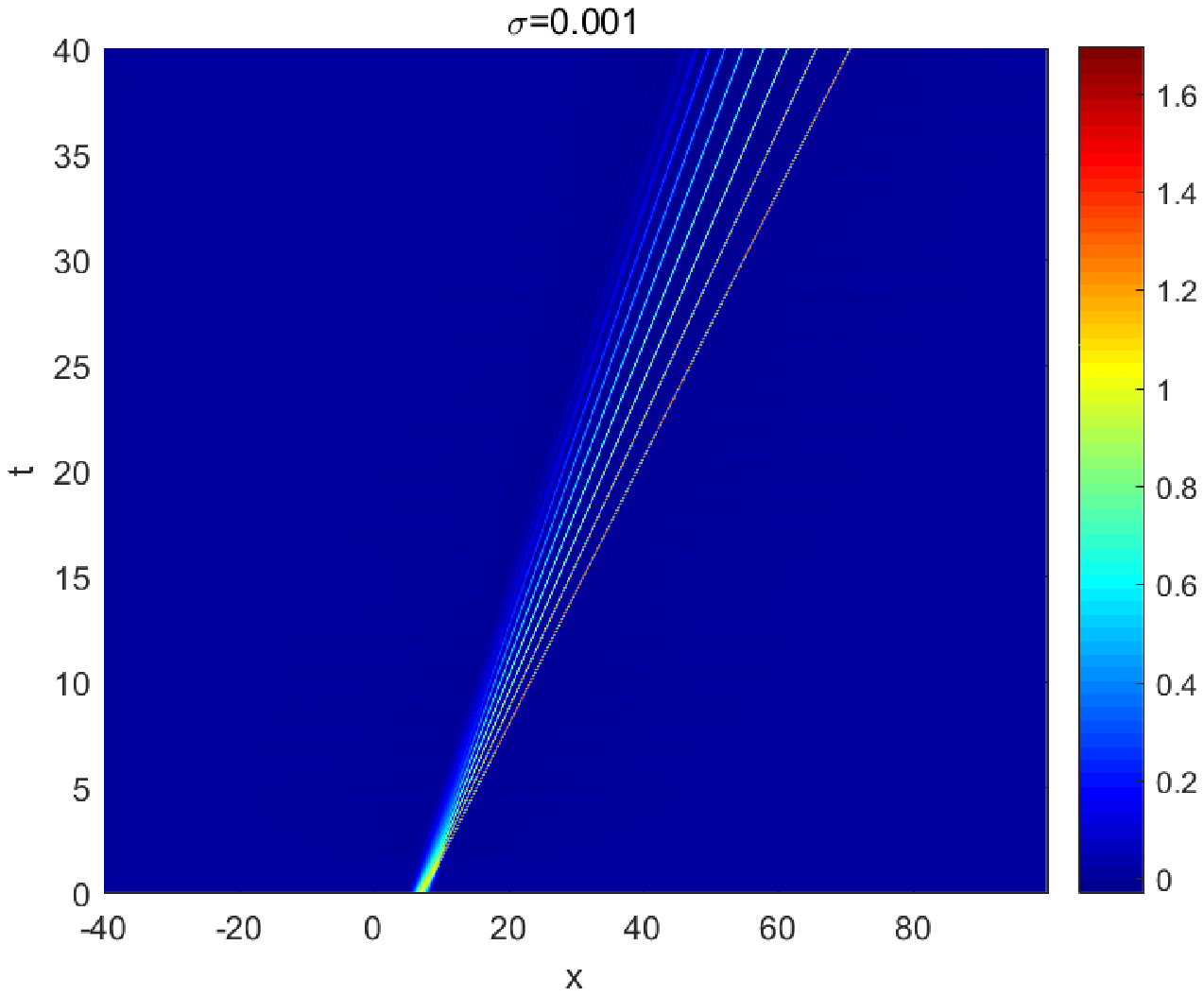}}\\
	\subfigure{
		\includegraphics[width=0.320\textwidth,height=0.30\textwidth]{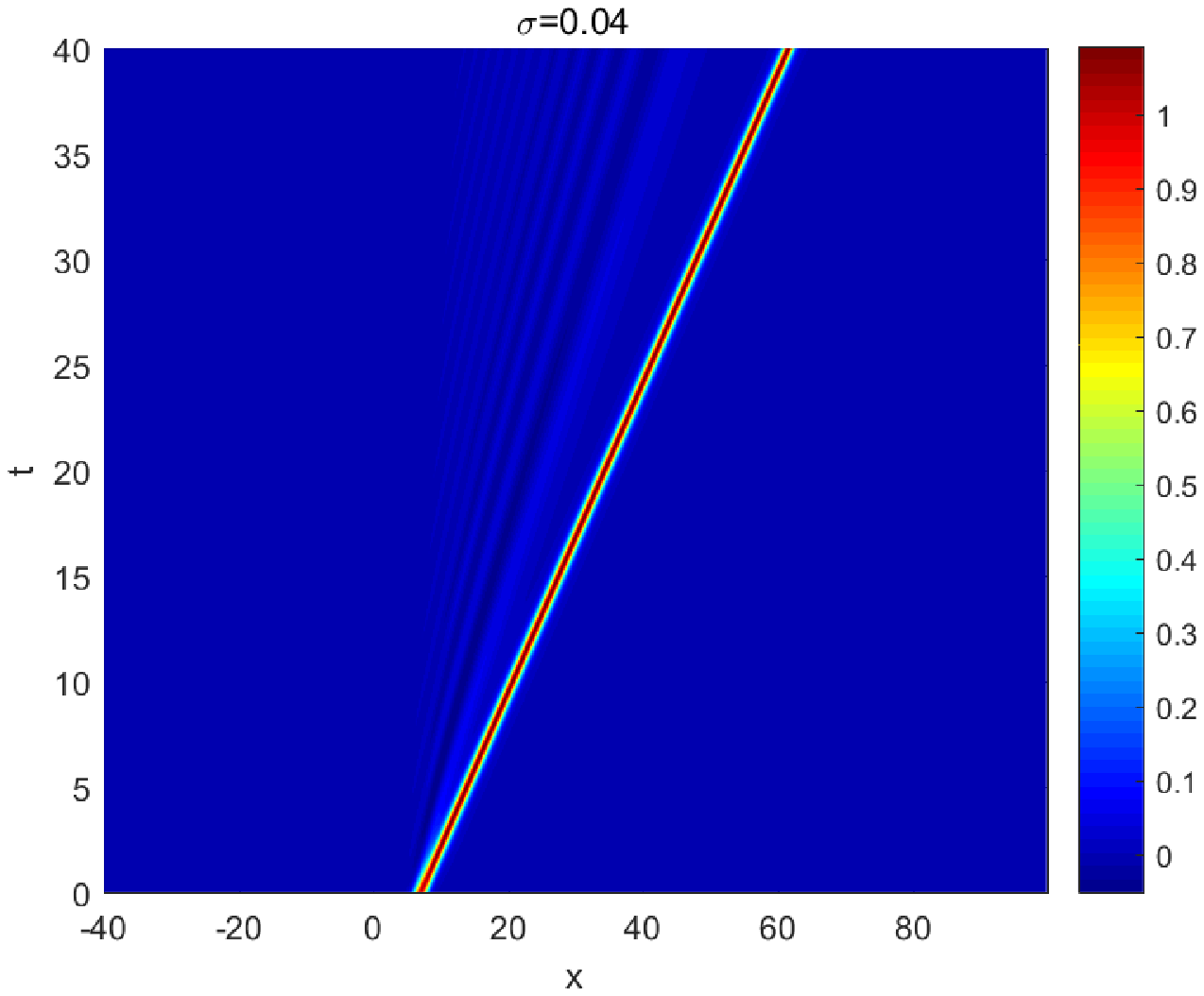}}
	\subfigure{
		\includegraphics[width=0.320\textwidth,height=0.30\textwidth]{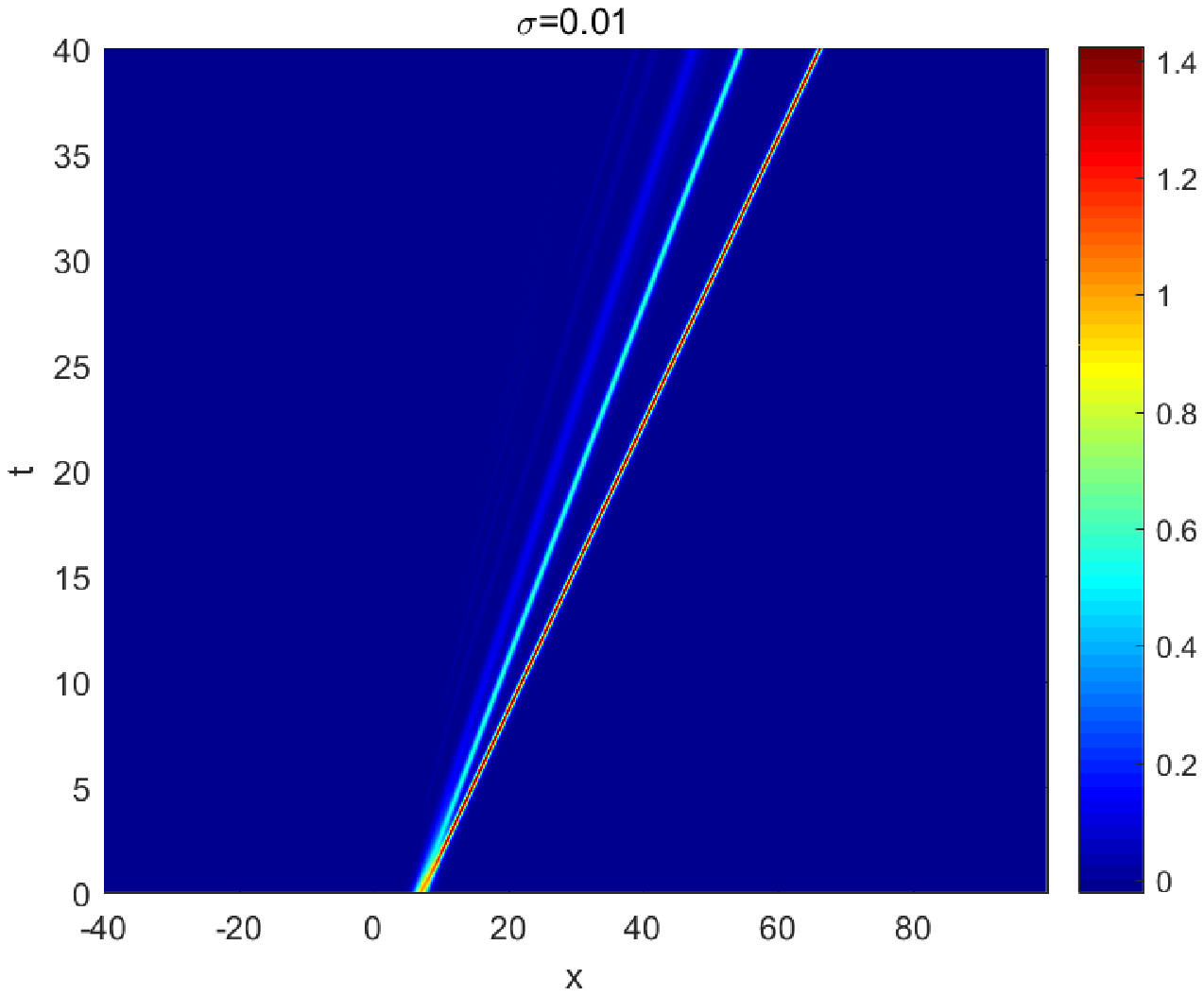}}
	\subfigure{
		\includegraphics[width=0.320\textwidth,height=0.30\textwidth]{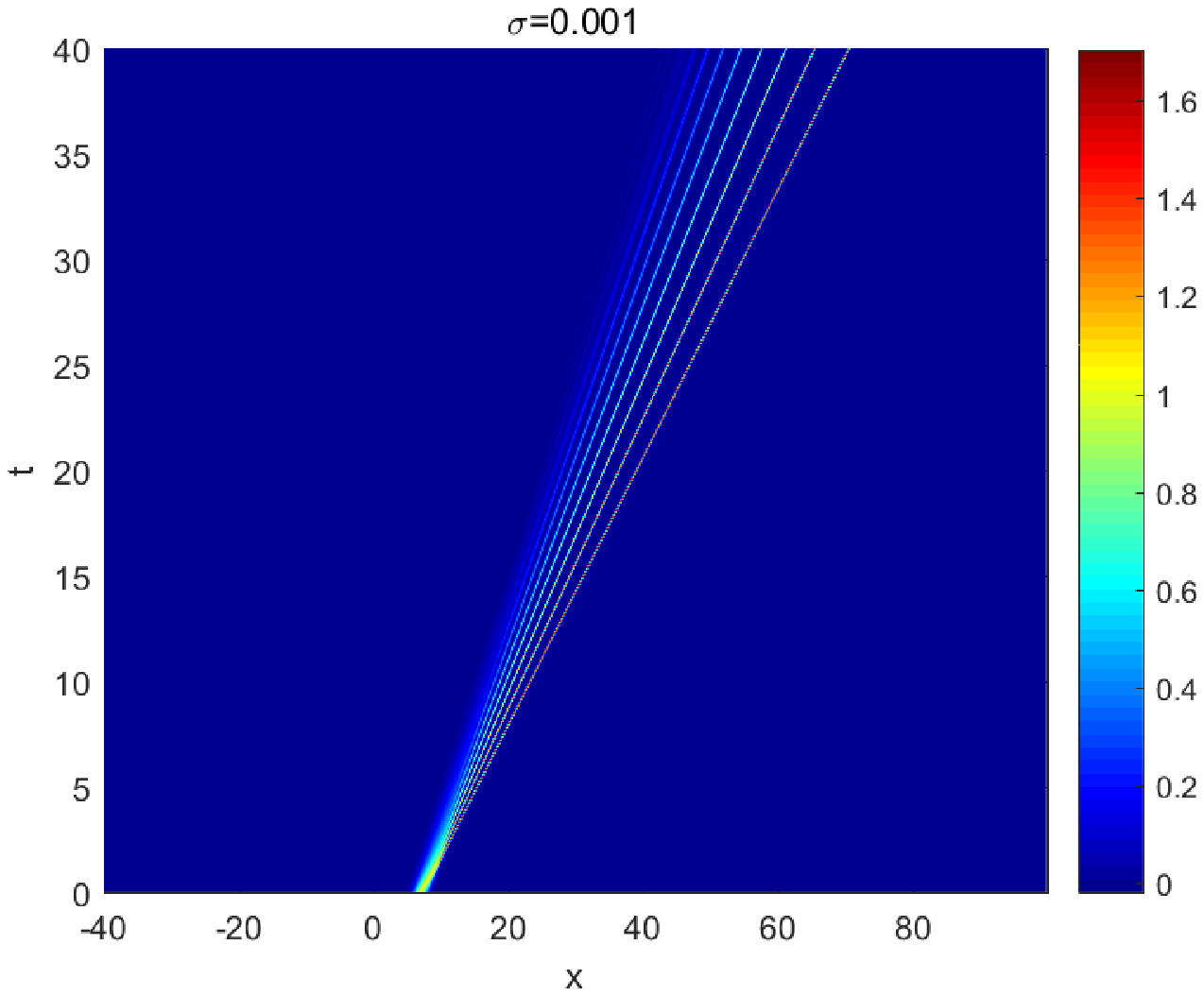}}
	\caption{\small The profiles of the solution with Maxwellian initial condition  computed by the schemes LCN-MP (top) and LLF-MP (bottom) at $T=40$. \label{fig:maxwellian-motion}}
\end{figure}

\begin{figure}[!htbp]
	\centering
	\subfigure{
		\includegraphics[width=0.30\textwidth,height=0.30\textwidth]{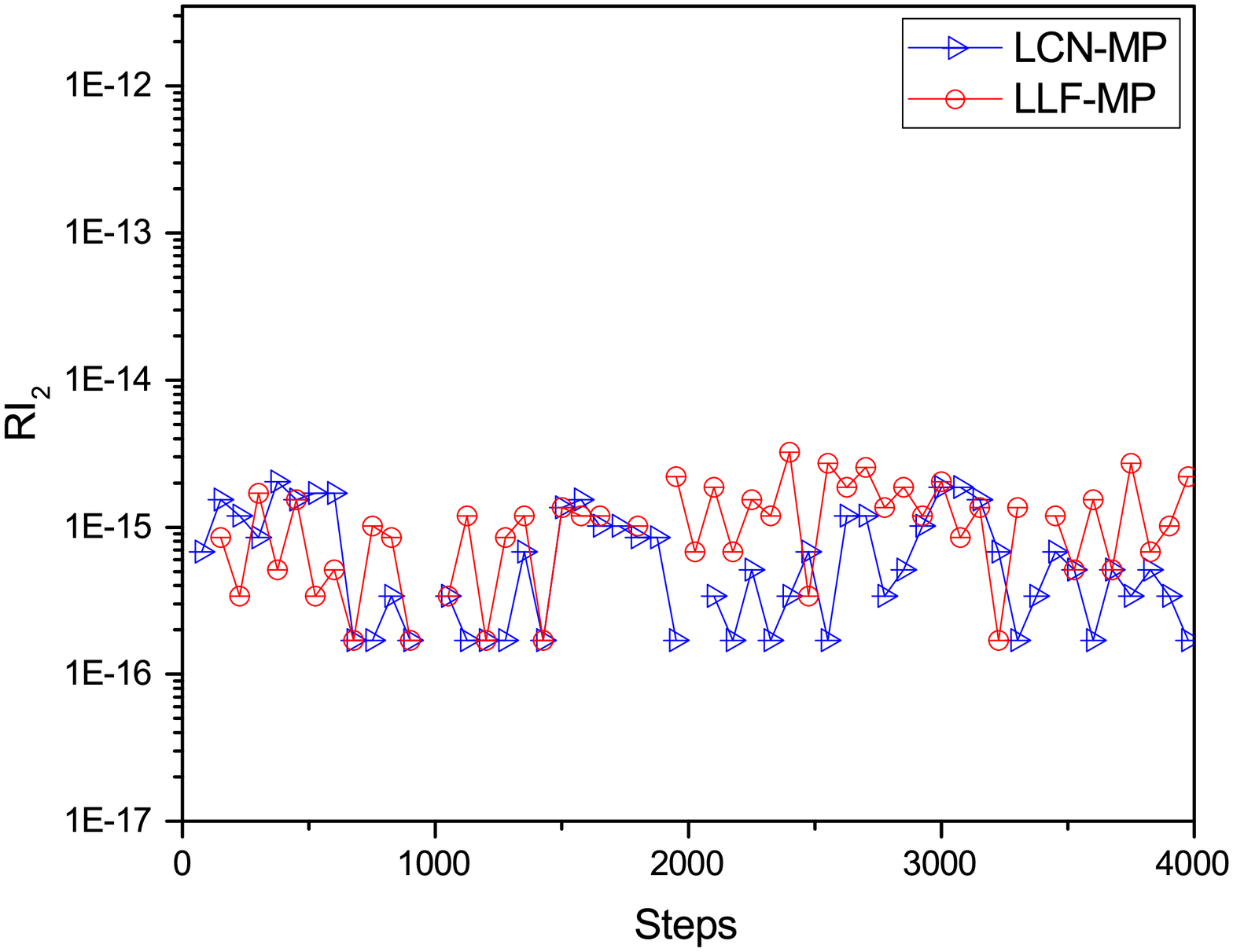}}
	\subfigure{
		\includegraphics[width=0.30\textwidth,height=0.30\textwidth]{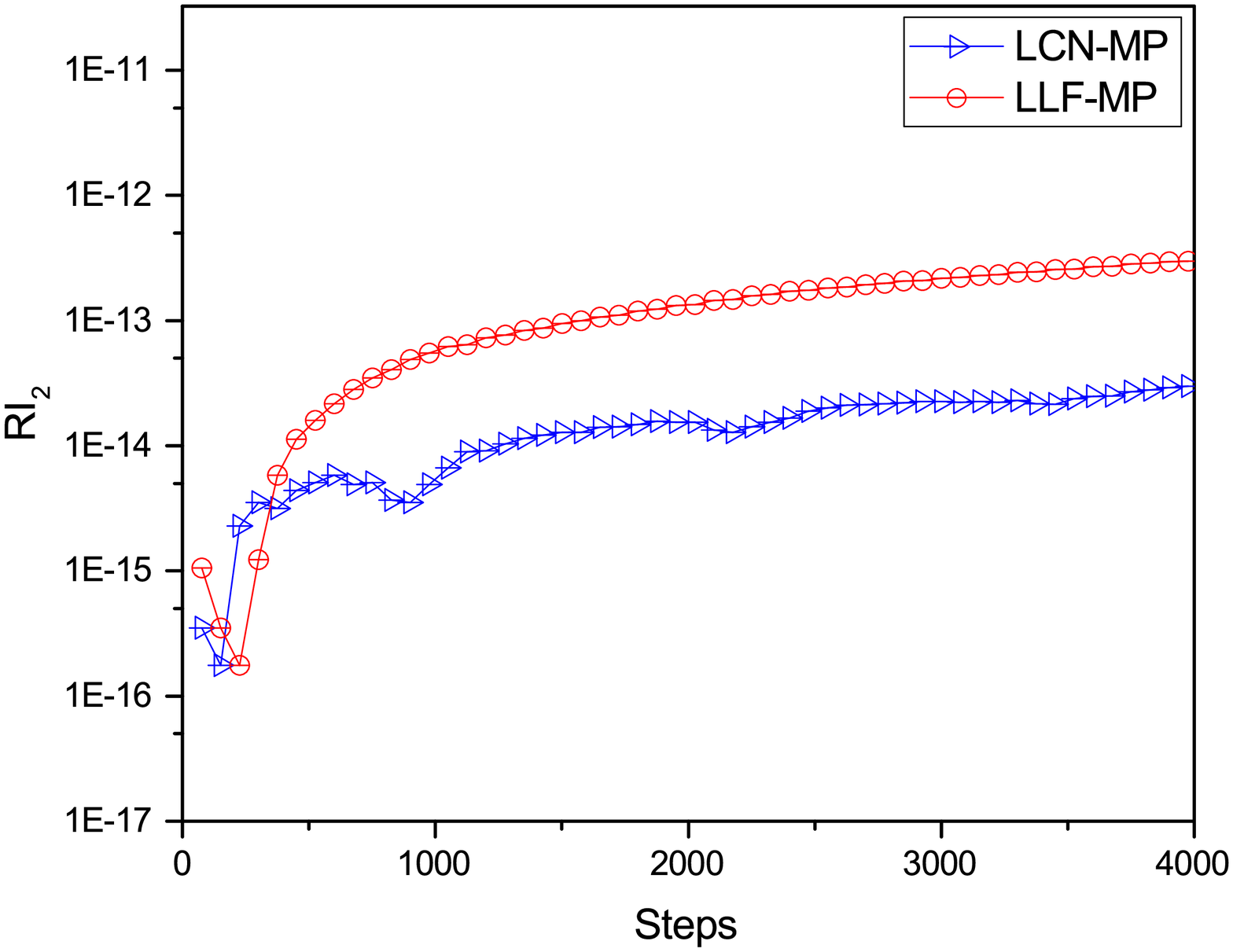}}
	\subfigure{
		\includegraphics[width=0.30\textwidth,height=0.30\textwidth]{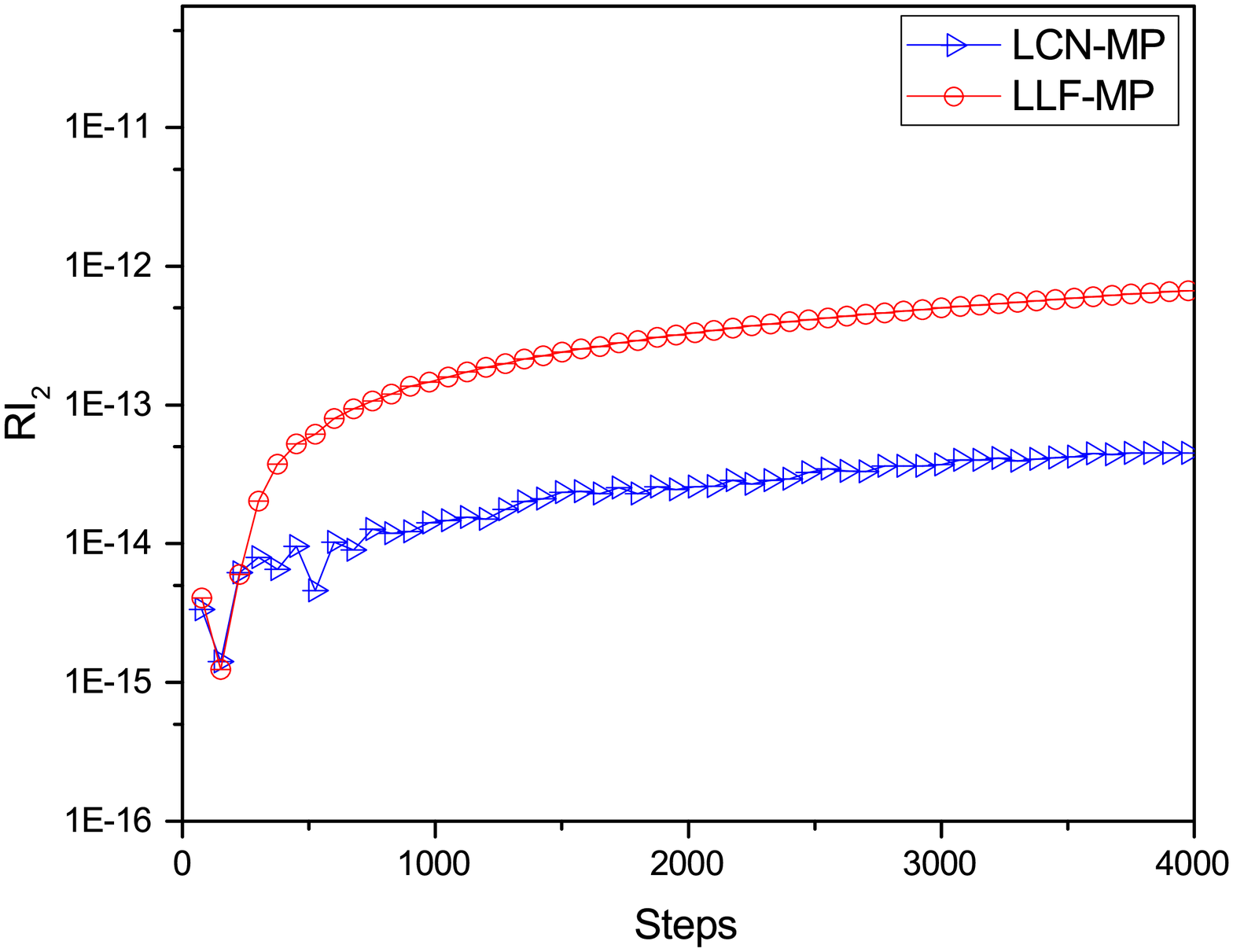}}
	\caption{\small The relative momentum error with  with $\sigma=0.04$ (left), $\sigma=0.01$ (middle) and $\sigma=0.001$ (right)  until $T=40$.\label{fig:maxwellian-momentum-0001}}
\end{figure}
In a word, these numerical results  confirm that the
convergence property as well as the efficiency and accuracy of the two new schemes. Moreover, they preserve the momentum very well.

\section{ Concluding remarks}
In this paper, we have developed two fully discrete linear-implicit conservative Fourier pseudo-spectral schemes for the RLW equation, including a linear-implicit Crank-Nicolson Fourier pseudo-spectral scheme (LCN-MP) and a linear-implicit leap-frog method (LLF-MP). The proposed schemes are proved to conserve the discrete momentum conservation law and be uniquely solvable. In addition, they are linear, i.e., only a linear equation system needs to be solved at each time step. The FFT algorithm is also used to speed up the computation in the numerical implementation. We utilize the standard energy method to prove in detail that LCN-MP is convergent in the order of $\mathcal{O}(\tau^2+N^{-r})$ in the discrete $L^{\infty}$ norm. The analysis technique can also be used easily to prove the convergence of LLF-MP. Numerical experiments are presented to illustrate the excellent performance of the proposed schemes in the end.

Overall, these two linear conservative schemes are accurate and efficient, and the idea presented in this paper can be readily extended to study a broader class of Hamiltonian PDEs for developing momentum-preserving algorithms. Note that the linear conservative schemes can be generalized naturally into the multi-dimensional case. But the current analytical technique is no longer valid in high dimensions, which will be considered in the future work.

\end{document}